\documentclass[12pt,reqno]{amsart}

\usepackage{graphicx}
\usepackage{epstopdf}
\usepackage{epic}
\usepackage{url}
\usepackage{latexsym}
\usepackage{amsmath,amsthm,amsfonts,
amssymb}
\usepackage {float}
\newcommand\mycom[2]{\genfrac{}{}{0pt}{}{#1}{#2}}
\newtheorem{theorem}{Theorem}[section]

\newtheorem{proposition}{Proposition}[section]
\newtheorem{prop}{Proposition}[section]
\newtheorem{lemma}{Lemma}[section]
\newtheorem{definition}{Definition}[section]
\newtheorem{corollary}{Corollary}[section]

\newtheorem{problem}[]{Problem}
\newtheorem{example}{Example}

\newtheorem*{remark}{Remark}

\newcommand{\g}{{\gamma}}
\newcommand{\s}{{\sigma}}
\newcommand{\Des}{\operatorname{Des}}

\newcommand{\uu}{\textbf{u}}

\def\S{{\mathfrak S}}

\def\Z{\mathbb{Z}}
\def\N{\mathbb{N}}

\def\a{\alpha}
\def\be{\beta}

\def\asc{\mathrm{asc}}
\def\des{\mathrm{des}}

\def\exc{\mathrm{exc}}
\def\cyc{\mathrm{cyc}}
\def\fix{\mathrm{fix}}
\def\drop{\mathrm{drop}}

\def\lrmaxda{\mathrm{lrmaxda}}

\def\rlmaxdd{\mathrm{rlmaxdd}}

\def\lda{\mathrm{da}}
\def\rdd{\mathrm{dd}}

\def\lrmax{\mathrm{lmax}}
\def\rlmax{\mathrm{rmax}}

\def\fmax{\mathrm{lmaxpk}}
\def\bmax{\mathrm{rmaxpk}}
\def\lrmaxda{\mathrm{lmaxda}}
\def\rlmaxdd{\mathrm{rmaxdd}}

\def\fmin{\mathrm{bmin}}

\def\S{{\mathfrak S}}

\def\basc{\textrm{basc}}
\def\suc{\textrm{suc}}

\def\lrmin{\mathrm{lrmin}}
\def\rlmin{\mathrm{rlmin}}
\def\LRmin{\mathrm{lmin}}
\def\RLmin{\mathrm{rmin}}
\def\LRmax{\mathrm{lmax}}
\def\RLmax{\mathrm{rmax}}
\def\PRW{\mathcal{M}}

\def\DD{\mathrm{D}}
\def\UD{\mathrm{Alt}}

\def\cyc{\mathrm{cyc}}
\def\fix{\mathrm{fix}}
\def\drop{\mathrm{drop}}
\def\cdd{\mathrm{cdd}}
\def\cda{\mathrm{cda}}
\def\cpk{\mathrm{cpk}}
\def\cval{\mathrm{cval}}
\def\dd{\mathrm{dd}}
\def\da{\mathrm{da}}

\def\Cdd{\mathrm{Cdd}}
\def\Cda{\mathrm{Cda}}
\def\Cpk{\mathrm{Cpk}}
\def\Cval{\mathrm{Cval}}
\def\Fix{\mathrm{Fix}}
\def\Drop{\mathrm{Drop}}
\def\Exc{\mathrm{Exc}}
\def\Orb{\mathrm{Orb}}

\def\val{\mathrm{val}}
\def\pk{\mathrm{pk}}
\def\da{\mathrm{da}}
\def\dd{\mathrm{dd}}

\def\Drop{\mathrm{Drop}}
\def\Exc{\mathrm{Exc}}
\def\Fix{\mathrm{Fix}}

\def\Suc{\mathrm{ Suc}}
\def\Basc{\mathrm{ Basc}}
\def\Des{\mathrm{ Des}}

\def\blue{\textcolor{blue}}
\def\red{\textcolor{red}}

\def\S{\mathfrak{S}}
\def\A{\mathcal{A}}
\def\T{\mathcal{T}}
\def\Z{\mathbb{Z}}
\def\N{\mathbb{N}}

\def\g{\gamma}
\def\s{\sigma}
\def\a{\alpha}

\def\asc{\mathrm{asc}}
\def\des{\mathrm{des}}

\def\exc{\mathrm{exc}}
\def\drop{\mathrm{drop}}
\def\fix{\mathrm{fix}}
\def\cyc{\mathrm{cyc}}
\def\da{\mathrm{da}}
\def\dd{\mathrm{dd}}
\def\val{\mathrm{val}}
\def\pk{\mathrm{pk}}

\def\Des{\mathrm{Des}}

\def\AC{\mathcal{AC}}
\def\lrmax{\mathrm{lmax}}
\def\rlmax{\mathrm{rmax}}
\def\lrmaxda{\mathrm{lmaxda}}
\def\rlmaxdd{\mathrm{rmaxdd}}
\def\rlmin{\mathrm{rmin}}
\def\rlminda{\mathrm{rminda}}
\def\rsho{\mathrm{rface}}
\def\leaf{\mathrm{leaf}}
\def\and{\mathrm{d}}
\def\Web{\mathcal{CA}}

\makeatletter
\newsavebox{\@brx}
\newcommand{\llangle}[1][]{\savebox{\@brx}{\(\m@th{#1\langle}\)}%
  \mathopen{\copy\@brx\kern-0.5\wd\@brx\usebox{\@brx}}}
\newcommand{\rrangle}[1][]{\savebox{\@brx}{\(\m@th{#1\rangle}\)}%
  \mathclose{\copy\@brx\kern-0.5\wd\@brx\usebox{\@brx}}}
\makeatother

\usepackage{geometry}
\usepackage{mathtools}

\usepackage{xparse}
\usepackage{caption}

\numberwithin{equation}{section}
\date{\today}
\usepackage{color}
\usepackage[dvipsnames]{xcolor}

\usepackage{tikz}
\usetikzlibrary{arrows}
\usepackage{microtype}
\usepackage{colortbl}
\usepackage{xparse}
\usepackage{caption}

\def\and{\mathrm{d}}
\usepackage[pagebackref]{hyperref}
\allowdisplaybreaks[4]

\begin{document}

\title[Gamma positivity of variations of $(\a,t)$- 
Eulerian polynomials]
{Gamma positivity of variations of $(\a,t)$-Eulerian polynomials}

\author{Chao Xu and Jiang Zeng}
\address{Universite Claude Bernard Lyon 1, CNRS UMR5208, Institut Camille Jordan\\
F-69622, Villeurbanne, France}
\email{xu@math.univ-lyon1.fr, zeng@math.univ-lyon1.fr}


\begin{abstract} 
In 1977 Carlitz and Scoville introduced the cycle 
$(\alpha,t)$-Eulerian polynomials  $A^{\cyc}_n(x,y, t\,|\,\alpha)$ by 
enumerating permutations with respect to the number of excedances, drops, fixed points and cycles. In this paper,
we introduce a nine-variable  generalization  of the Eulerian polynomials $A_n(u_1,u_2,u_3,u_4, f, g,  t\,|\,\alpha, \beta)$ in terms of  descent based statistics of permutations and prove a connection formula 
 between these two kinds of generalized Eulerian polynomials.
 By exploring the connection formula,
 we derive plainly the exponential generating function of the latter polynomials and various 
 $\gamma$-positive  formulas for  variants of Eulerian polynomials.  In particular, 
our results unify and strengthen the recent 
 results   by Ji  and  Ji-Lin. 
 In related work to the transition matrix between the Specht and web bases, Hwang, Jang and Oh recently introduced 
 the web permutations, which can be characterised by  cycle  André permutations.
 We show that enumerating the latter permutations with respect to the number of drops, fixed points and cycles gives rise to the normalised $\gamma$-vectors of the $(\alpha,t)$-Eulerian polynomials. Our result generalizes and unifies several known results in the literature. 
\end{abstract}


\maketitle
\tableofcontents
\section{Introduction}
The Eulerian polynomials $A_n(x)$  are defined by the exponential generating function~\cite{FS70,Pe15}
\begin{equation}\label{eulerian polynomila-def}
\sum_{n\geq 0} A_n(x)\frac{z^n}{n!}=\frac{1-z}{e^{(z-1)x}-z}.
\end{equation}
 These  polynomials  have a long 
and rich history, some of which is given in   Foata's  survey~\cite{Fo10}, Stanley's EC~\cite{St12,St99} and Petersen's book~\cite{Pe15}, see also  
Foata and Schützenberger's monograph~\cite{FS70}, which remains valuable and continues to provide useful insights and information.  
  For any positive integer $n$, 
we denote  the symmetric group of    $[n]:=\{1, 2,\ldots, n\}$ by $\S_n$.
For  $\s\in \S_n$, 
 an index  $1\leq i\leq n-1$ is called
 \begin{itemize}
 \item a \emph{descent} (\textbf{des}) of $\s$ if $\s(i)>\s(i+1)$;
 \item an \emph{ascent} (\textbf{asc}) of $\s$ if $\s(i)<\s(i+1)$;
\item  an \emph{excedance} (\textbf{exc}) of $\s$ if $i<\s(i)$.
 \end{itemize}
The following 
combinatorial interpretations of $A_n(x)$ are well-known:
\begin{align}
A_n(x)=\sum_{\s\in \S_n}x^{\asc(\s)}=\sum_{\s\in \S_n}x^{\des(\s)}=\sum_{\s\in \S_n}x^{\exc(\s)}.
\end{align}


The binomial-Eulerian polynomials, which are the 
h-polynomials of stellohedrons, are defined by
\begin{equation}
\tilde{A}_n(x):=\sum_{\sigma\in \PRW_{n+1}} x^{\des(\sigma)}=1+x\sum_{m=1}^n\binom{n}{m}A_m(x),
\end{equation}
where $\PRW_n$ is the set of permutations $\s\in\S_n$ such that the first descent (if any) of $\s$ appears at the letter $n$. 
For instance,
  we have 
\[
\PRW_2=\{12,21\} \quad \textrm{and}\quad
 \PRW_3=\{123,132,231, 312,321\}.
\]

Any polynomial with real coeffcients $h(x)=\sum_{i=0}^n h_i x^i$ 
satisfying $h_i=h_{n-i}$ can be expressed uniquely in the form
$h(x)=\sum_{k=0}^{\lfloor n/2\rfloor}\gamma_kx^k(1+x)^{n-2k}$. The coefficients $\gamma_k$  are called the $\gamma$-coefficients of $h(x)$. If the $\gamma$-coefficients of $h(x)$ are all nonnegative, then $h(x)$ is said to be $\gamma$-positive. It is often of interest in these cases to find combinatorial interpretations to the $\gamma$-coefficients. 
It is known~\cite{FS70,PZ21}  that the Eulerian polynomials $A_n(x)$ have the following 
$\gamma$-positive expansion 
 \begin{equation}\label{euler-gamma}
 A_{n+1}(x)=
 \sum_{j=0}^{\lfloor n/2\rfloor}2^j d_{n,j}\,x^j(1+x)^{n-2j},
 \end{equation}
 where $d_{n,j}$ is the number of Andr\'e permutations in $\S_{n+1}$ with $j$ descents.
We refer the reader to 
  \cite{LZ15, SW20, PZ21} for some recent  developments  in this direction and 
to \cite{At18, Pe15}  for a general  overview of the subject.
The gamma positivity of $\tilde{A}_n(x)$ follows from a general result of Postnikov, Reiner, and Williams~\cite[Section~10.4]{PRW08}, and implies that these polynomials are symmetric. The coefficients of these polynomials can also be beautifully expressed as sums of products of Eulerian numbers and binomial coefficients.  Let $A_n(x)=\sum_{k=0}^{n-1}\left\langle \mycom{n}{k}\right\rangle x^k$. 
In related work, for integers $a, b\geq 0$,  Chung, Graham, and Knuth~\cite{CGK10} found the following 
identity with three proofs:
\begin{equation}\label{cgk}
\sum_{k\geq 0} \binom{a+b}{k}
\left\langle \mycom{k}{a-1}\right\rangle=
\sum_{k\geq 0}\binom{a+b}{k}
\left\langle \mycom{k}{b-1}\right\rangle,
\end{equation}
where $\left\langle\mycom{0}{0}\right\rangle=1$.
Shareshian and Wachs~\cite{SW20} noticed that the above identity corresponds exactly 
to the palindromicity  of  the coefficients of $\tilde{A}_n(x)$. 

For 
$\s\in \S_n$,  
 an index  $i\in [n]$ is a
 \begin{itemize}
 \item  \emph{drop} (\textbf{drop}) of $\s$ if $i>\s(i)$;
\item  \emph{fixed point} (\textbf{fix}) of $\s$ if $i=\s(i)$.
 \end{itemize}
We shall also consider a permutation $\sigma\in \S_n$ as a word
$\sigma=\sigma_1\ldots \sigma_n$ with $\sigma_i:=\sigma(i)$ for
$i\in [n]$.
Say that a letter $\s_i$ is a
\begin{itemize}
    \item 
 \emph{left-to-right maximum} (\textbf{lrmax}) of $\s$
if $\s_i>\s_j$ for every $j<i$,
\item 
 \emph{right-to-left maximum} (\textbf{rlmax})  of $\s$ if $\s_i>\s_j$ for every  $j>i$.  
\end{itemize}

In the middle of 1970's 
Carlitz-Scoville considered  several multivariate 
 Eulerian polynomials,  among which  are the so-called $(\a, \be)$-Eulerian polynomials~\cite{CS74} 
 \begin{subequations} 
\begin{equation}\label{Carlitz-Scoville-alpha-beta-polynomials}
A_n(x,y\,|\, \a, \be):=\sum_{\s\in \S_{n+1}}
x^{\asc(\s)}y^{\des(\s)}\a^{\lrmax(\s)-1}\be^{\rlmax(\s)-1},
\end{equation}
and the  following  ones~\cite{CS77}, that we refer to \textbf{cycle $(\a,t)$-Eulerian polynomials},  
\begin{align}\label{CS:def1}
A^{\cyc}_n(x,y, t\,|\,\alpha):=\sum_{\s\in \S_n} x^{\exc(\s)} y^{\drop(\s)}t^{\fix(\s)}\alpha^{\cyc(\s)},
\end{align}
\end{subequations} 
where  $\cyc(\s)$ denotes 
the number of cycles of $\s$.
 The first values of these polynomials are
\begin{align*}
A_1^{\cyc}( x,y,t\,|\,\alpha)&= \a  t,\\
A_2^{\cyc}( x,y,t\,|\,\alpha)&=\a x y +\a^{2}  t^{2},\\
A_3^{\cyc}( x,y,t\,|\,\alpha)&=\a \left(x +y \right) x y +3 \a^{2}
 x y  t +\a^{3}  t^{3},\\
A_4^{\cyc}( x,y,t\,|\,\alpha)&=\a xy\left( \,x^{2}+\left(3 \a +4\right)  xy+ \,y^{2}\right)+4 \a^2xy (x +y ) t 
+6 \a^{3}   xy  t^{2}+\a^{4}  t^{4}.
\end{align*}
As  $A^{\cyc}_n(x,y, t\,|\,\alpha)=y^nA^{\cyc}_n(x/y,1, t/y\,|\,\alpha)$, 
the polynomials $A^{\cyc}_n(x,y, t\,|\,\alpha)$ are the homogeneous version of
$A^{\cyc}_n(x,1, t\,|\,\alpha)$, which are  studied   in \cite[Chapter 4]{FS70}.

For $\s=\sigma_1\ldots \sigma_n\in \S_n$
with  the boundary condition $\mathbf{0-0}$, i.e., 
$\mathbf{\s_0=\s_{n+1}=0}$, a letter  $\s_i\in[n]$ is called a
  \begin{itemize}
   \item  {\em valley} (\textbf{val}) of $\s$ if $\sigma_{i-1}>\sigma_{i}<\sigma_{i+1}$;
  \item  {\em peak} (\textbf{pk}) of $\s$ if $\sigma_{i-1}<\sigma_i>\sigma_{i+1}$;
  \item  {\em double ascent} (\textbf{da}) of $\s$ if $\sigma_{i-1}<\sigma_{i}<\sigma_{i+1}$;
  \item  {\em double descent} (\textbf{dd}) of $\s$ if $\sigma_{i-1}>\sigma_{i}>\sigma_{i+1}$.
  \end{itemize}
It is clear that  the following identities hold
 \begin{align}\label{relation:V and W}
\val=\pk-1,\quad \asc=\val+\lda,\quad \des=\val  +\rdd.
 \end{align} 
Recently,    
refining  the  $(\alpha,\beta)$-Eulerian polynomials $A_n(x,y\,|\, \a, \be)$,  
Ji~\cite{Ji23} 
  considered  a variation of Eulerian polynomials incorporating six  statistics over permutations in $\S_{n+1}$:
\begin{equation}\label{refine-CS-JI}
    A_n(u_1,u_2,u_3,u_4\,|\,\alpha, \beta)
:=\sum_{\sigma \in \mathfrak{S}_{n+1}}(u_1u_2)^{{\val}(\sigma)}u_3^{\lda(\sigma)}u_4^{\rdd(\sigma)}{\alpha}^{\lrmax(\sigma)-1}{\beta}^{\rlmax(\sigma)-1}.
\end{equation}
The  main result in~\cite[Theorem 1.4]{Ji23}   is a  formula for the  exponential generating function of $A_n(u_1,u_2,u_3,u_4\,|\,\alpha, \beta)$, which is proved  by 
the context-free grammar calculus.~\footnote{Ji~\cite[(3.35)-(3.37)]{Ji23} assumed that $\a$ and $\beta$ are integers in her proof. 
In the latter case Carlitz-Scoville gave a different combinatorial interpretation of the same generating function~\cite[Theorem 5]{CS74}.} This   
generating function is shown to be   connected with  a bunch of  enumerative  polynomials in the literature~\cite{Ji23}.  
In a subsequent paper~\cite{JL23}, using  the context-free grammar calculus and  group actions over permutations, Ji and Lin proved  
  the $\gamma$-positivity of two $\a$-analogues of  the  Eulerian polynomials and  an $\alpha$-analogue of 
Chung-Graham-Knuth's symmetric Eulerian identity~\cite{CGK10}, see also \cite{PRW08, SW20}.

 This paper originally arose from the desire to provide 
 an alternative approach to  Ji's generating function formula~\cite[Theorem 1.4]{Ji23} via known
 results in the literature, 
 which  led us to     
  the following connection formula:
 \begin{equation}\label{equ: CS-A_n^cyc}
 A_n(x,y\,|\,\a, \be)=A_n^{\cyc}\biggl(x,y,\frac{\a x+\be y}{\a+\be}\,|\, \a+\be\biggr),
 \end{equation} 
 and more generally
 \begin{equation}\label{CS-Ji connection}
 A_n(u_1,u_2,u_3,u_4\,|\,\alpha, \beta)=A_n^{\cyc}\biggl(x,y,\frac{\a u_3+\be u_4}{\a+\be}\,|\, \a+\be\biggr),
 \end{equation}
 where $xy=u_1u_2$ and $x+y=u_3+u_4$.

 The generating function of $A_n(u_1,u_2,u_3,u_4\,|\,\alpha, \beta)$ then  follows 
 by combining \eqref{CS-Ji connection} 
 with the known generating function of $A^{\cyc}_n(x,y, t\,|\,\alpha)$,  which has several known combinatorial proofs. Actually, we will generalize 
Ji's polynomial  with three more variables   
in order to unify the  two families of 
 polynomials in~\cite{Ji23,JL23} as well as 
  some recent  results in \cite{JL23, M2Y224,MQYY24}. 

For  a permutation $\s=\sigma_1\ldots \sigma_n\in \S_n$, we say that an index
 $i\in [n]$ is a
\begin{itemize}
\item  \emph{cycle peak} (\textbf{cpk}) of $\s$ if $\sigma^{-1}(i)<i>\sigma(i)$; 
\item  \emph{cycle valley} (\textbf{cval}) of $\s$ if $\sigma^{-1}(i)>i<\sigma(i)$; 
\item  \emph{cycle double ascent}  (\textbf{cda}) of $\s$ if $\sigma^{-1}(i)<i<\sigma(i)$; 
\item  \emph{cycle double descent}  (\textbf{cdd}) of $\s$ if $\sigma^{-1}(i)>i>\sigma(i)$. 
\end{itemize}
Note  that $\cpk(\s)=\cval(\s)$.
The following is our  first main result.
\begin{theorem}\label{master1} If $xy=u_1u_2$ and $x+y=u_3+u_4$, then
\begin{equation}\label{eq:master theorem}
A_n^{\cyc}(x,y, t\,|\,\alpha)= 
\sum_{\s\in \S_n} (u_1 u_2)^{\cpk(\s)}u_3^{\cda(\s)}u_4^{\cdd(\s)}t^{\fix(\s)}\alpha^{\cyc(\s)}.
\end{equation}
\end{theorem}

Let  $\sigma=\sigma_1\ldots \sigma_n$ be a permutation in 
$\S_n$ with  the boundary condition $\mathbf{0-0}$.  A letter $\s_i\in[n]$ is a
\begin{itemize}
    \item  \emph{left-to-right-maximum-peak} (\textbf{lmaxpk})  if 
 $\s_i$ is a left-to-right maximum and also a peak;
  \item  \emph{right-to-left-maximum-peak} (\textbf{rmaxpk})  if $\s_i$ is a  right-to-left maximum and also a peak;
\item  \emph{left-to-right-maximum-double-ascent} (\textbf{lmaxda}) if 
 $\s_i$ is a left-to-right  maximum and also  a double ascent;
  \item   \emph{right-to-left-maximum-double-descent} (\textbf{rmaxdd}) if 
   $\s_i$ is a right-to-left maximum and also 
  a double descent.  
  \end{itemize}

Let $\uu=(u_1,u_2,u_3,u_4)$ and define the generalized 
Eulerian polynomial
\begin{subequations} 
\begin{multline}\label{refine-JI2}
A_n(\uu, f, g,  t\,|\,\alpha, \beta)=\sum_{\sigma \in \mathfrak{S}_{n+1}}(u_1u_2)^{{\val}(\sigma)}u_3^{\lda(\sigma)}u_4^{\rdd(\sigma)}f^{\fmax(\s)-1}
g^{\bmax(\s)-1}\\
\times t^{\lrmaxda(\s)+\rlmaxdd(\s)}
{\alpha}^{{\lrmax}(\sigma)-1}{\beta}^{{\rlmax}(\sigma)-1}.
\end{multline}
  The following  is our second main result, which generalizes \eqref{CS-Ji connection}.
\begin{theorem}\label{master2}
If $xy=u_1u_2$ and $x+y=u_3+u_4$, then
\begin{equation}\label{equ:A(u,s,a,b)-A^cyc(x,y,t,a)-with-s}
A_n(\uu, f, g,  t\,|\,\alpha, \beta)=A_n^{\cyc}\left( x,y, \frac{\a u_3+\beta u_4}{\a f+\beta g}\,t\,|\, \a f+\beta g\right).
\end{equation}

\end{theorem}

As the exponential generating function of $A_n^{\cyc}(x,y, t\,|\,\alpha)$ is known, the above theorem implies immediately 
the following generating function, which is our third main result.

\begin{theorem}\label{master3} 
Let $xy=u_1u_2$ and  $x+y=u_3+u_4$. We have 
\begin{equation}\label{expgen:P_n-2}
\sum_{n\geq 0}A_{n}(\uu, f,g,t\,|\, {\alpha},{\beta})\frac{z^n}{n!}
 =e^{(\a u_3+\be u_4)tz}\left(\frac{x-y}{xe^{yz}-ye^{xz}}\right)^{\alpha f+\be g}.
 \end{equation}

\end{theorem}
\end{subequations}

The rest of this paper is organized as follows. In Section~2.1
we review some known generating functions and prove Theorem~\ref{master3} by combining Theorems~\ref{master1} and \ref{master2}. In Sections~2.2-2.5, 
we introduce 
$(\a,t)$-Eulerian polynomials $A_n(x,y, t\,|\,\alpha)$ and $(\a,t)$-binomial-Eulerian polynomials $\widetilde{A}_n(x,y,t\,|\,\alpha)$
and prove that these polynomials  have  interesting $\gamma$-expansions, which  include   $t$-analogues of Ji-Lin's results~\cite[Theorems 1.5 and 1.6]{JL23}. Moreover,  we present an $(\a,t)$-analogue of 
\eqref{euler-gamma} with 
three combinatorial interpretations for the  $(\a,t)$-analogue of $d_{n,j}$ over cycle André permutations and André permutations of types 1 and 2 (Theorem~\ref{main-theorem1}). 
In Section~3, we prove the main results in Theorems~\ref{master1} and \ref{master2}. Our proof uses Foata’s (first) fundamental transformation (FFT)~\cite[Chapter 10]{Lo83} and a cyclic valley-hopping  on permutations due to  Cooper, Jones and Zhuang~\cite{CJZ20}. Previously, 
Athanasiadis-Savvidou~\cite{AS12} and  Sun-Wang~\cite{SW14} used the cyclic valley-hopping  on derangements. 
 In Section~4, we will prove Theorem~\ref{main-theorem1} about  three combinatorial interpretations for 
 the normalised $\gamma$-coefficients of $(\a,t)$-Eulerian polynomials in the models of cycle André permutations and André permutations of type 1 and type 2.
 In Section~5, by a  variant of FFT, we link  the polynomial $A_n^{\cyc}( x,y,t\,|\,\alpha)$ with the 
enumeration of successions of permutations and  compute  the exponential generating function of a multivariate statistic involving the number of successions of permutations.
Finally, we close the paper with some open problems.
\section{Applications of the main theorems}  
\subsection{Exponential  generating functions: Proof of Theorem~\ref{master3}}
The exponential generating function of polynomials
$ A_n^{\cyc}(x,y, t\,|\,\alpha)$ is well-known and reads as follows
\begin{subequations}
\begin{equation}\label{equ:gen-CS-cyc}
\sum_{n\geq 0} A_n^{\cyc}(x,y, t\,|\,\alpha)\frac{z^n}{n!}=
\left(\frac{(x-y)e^{tz}}{xe^{yz}-ye^{xz}}\right)^\alpha.
\end{equation}
An equivalent form of formula~\eqref{equ:gen-CS-cyc} was 
proved in  Foata-Sch\"utzenberger's monograph~\cite[Chapitre IV  (12) and (13)]{FS70} 
by using the exponential formula.
This  formula  was later 
proved  by Cartliz-Scoville~\cite[Theorem 1]{CS77} by  calculus, see also 
Brenti~\cite[Proposition~7.3]{Br00} for a proof using  symmetric functions and plethysm.  
 Actually, 
in view of the exponential formula~\cite{FS70, St99}, Eq.~\eqref{equ:gen-CS-cyc} is  equivalent to either of the  following exponential generating functions for derangements:
\begin{align}\label{kim-zeng}
\sum_{n=0}^\infty 
A_n^{\cyc}(x,1, 0\,|\, 1)\frac{z^n}{n!}
&=
\biggl(1-\sum_{n=2}^\infty (x+x^2+\cdots +x^{n-1})\frac{z^n}{n!}\biggr)^{-1},\\
\sum_{n=0}^\infty A_n^{\cyc}(x,1, 0\,|\, -1)\frac{z^n}{n!}&=
1-\sum_{n=2}^\infty (x+x^2+\cdots +x^{n-1})
\frac{z^n}{n!}. \label{ksavrelov-zeng}
\end{align}
\end{subequations}
 Kim and Zeng~\cite{KZ01} provide  a  bijective proof of \eqref{kim-zeng}, of which    a
nice proof was given by Gessel~\cite{Ge91} in the case where $x=1$, 
 and  two  combinatorial proofs of
 \eqref{ksavrelov-zeng} are given  in~\cite{KZ02}.

The following refinement of \eqref{equ:gen-CS-cyc}  was proved   in ~\cite[Théorème 1]{Ze93}:
\begin{align}\label{equ:gen-JZ-cyc}
1+\sum_{n\geq 1}\frac{z^n}{n!} \sum_{\s\in \S_n} (u_1 u_2)^{\cpk(\s)}u_3^{\cda(\s)}u_4^{\cdd(\s)}t^{\fix(\s)}\alpha^{\cyc(\s)}=\left(\frac{(x-y)e^{tz}}{xe^{yz}-ye^{xz}}\right)^\alpha,
\end{align}
where $xy=u_1u_2$ and $x+y=u_3+u_4$.  
Clearly, formula~\eqref{equ:gen-JZ-cyc} implies \eqref{equ:gen-CS-cyc} and then Theorem~\ref{master1}. 
Reversely, combining Theorem~\ref{master1} and \eqref{equ:gen-CS-cyc} yields \eqref{equ:gen-JZ-cyc}.
\begin{subequations}
\begin{proof}[Proof of Theorem~\ref{master3}] Combining Theorem~\ref{master2}  with 
\eqref{equ:gen-CS-cyc}, we  derive immediately the exponential generating function of  $A_n(\uu,f,g,t\,|\,\a,\be)$  in \eqref{refine-JI2}, namely
\begin{align}\label{master-egf}
\sum_{n\geq 0}A_{n}(\uu, f,g,t\,|\, {\alpha},{\beta})\frac{z^n}{n!}
 =e^{(\a u_3+\be u_4)tz}\left(\frac{x-y}{xe^{yz}-ye^{xz}}\right)^{\alpha f+\be g},
 \end{align}
 where $xy=u_1u_2$ and $x+y=u_3+u_4$.
 \end{proof}

The polynomials 
 $A_{n}(\uu,1,1,1\,|\, {\alpha},{\beta})$ and $A_n(\uu, 0, 1,1\,|\, {\alpha},{\beta})$ correspond, respectively,  to 
 Ji's generalized Eulerian polynomial~\eqref{refine-CS-JI} and Ji-Lin's binomial-Stirling Eulerian polynomial~\cite[Theorem 1.5]{JL23}.  
\begin{corollary} We have 
\begin{align}\label{equ:generating functions of binomial}
\sum_{n\geq 0} A_n(\uu, 0, 1,1\,|\, {\alpha},{\beta})\frac{z^n}{n!} &=e^{(\a u_3+\beta u_4 )z}\left(\frac{x-y}{xe^{yz}-ye^{xz}}\right)^\beta,\\
\label{equivalent JI}
 \sum_{n\geq 0}A_{n}(\uu,1,1,1\,|\, {\alpha},{\beta})\frac{z^n}{n!}
 &=e^{(\a u_3+\be u_4)z}\left(\frac{x-y}{xe^{yz}-ye^{xz}}\right)^{\alpha +\be},
\end{align}
where $xy=u_1u_2$ and $x+y=u_3+u_4$.  
\end{corollary}
\end{subequations}

Note that \eqref{equivalent JI} is equivalent to Ji's Theorem 1.4~\cite{Ji23}.
Combining Theorem~\ref{master2}  and the continued fraction expansion of the 
ordinary generating function of $A_n^{\cyc}(x,y, t\,|\,\alpha)$ in  
\cite[Th\'eor\`eme 3]{Ze93} we obtain  the following continued fraction formula for the ordinary generating function of
$A_{n}(\uu, f,g,t\,|\, {\alpha},{\beta})$.
\begin{theorem}\label{con-frac-A}
We have
\begin{subequations}   
    \begin{equation}
\label{fix-carlitz-scoville-cf}
  \sum_{n=0}^\infty A_n(\uu,f,g,t\,|\,\alpha, \beta)z^n
 = \cfrac{1}{1-b_0 z-\cfrac{\lambda_1z^2}{1-b_1z-\cfrac{\lambda_2 z^2}{1-b_2z-\cdots}}},
\end{equation}
where
\begin{align}
b_k&=k(u_3+u_4)+(\a u_3+\beta u_4)t,\\
\lambda_{k+1}&=(k+\a f+\beta g)(k+1) u_1u_2\quad (k\geq 0).
\end{align}
\end{subequations}

\end{theorem}

\subsection{$(\a,t)$-Eulerian and binomial-Eulerian polynomials}
Define the   $(\a,t)$-Eulerian polynomials $A_n(x,y, t\,|\,\a)$ by
\begin{subequations}
\begin{align} \label{t-Ji-Lin}   
A_n\left(x,y,t\,|\, \a\right):
=\sum_{\s\in\S_{n+1}}x^{\asc(\s)}y^{\des(\s)} t^{\lrmaxda(\s)+\rlmaxdd(\s)}\a^{\LRmax(\s)+\RLmax(\s)-2},
\end{align}
which is equal to $A_{n}(x, y, x, y, 1,1,t\,|\, \a, \a)$. By Theorem~\ref{master2} we have
\begin{equation}\label{An=Ancyc}
    A_n\left(x,y,t\,|\, \a\right)=A_n^{\cyc}\left(x,y,\frac{x+y}{2}t\,|\, 2\a\right).
\end{equation}
Combining \eqref{An=Ancyc} and   
Theorem~\ref{master2} with $f=g=1$, $\a=\beta$, we obtain   the following $t$-analogue of   Ji-Lin's Theorem 1.6~\cite{JL23}.

\begin{theorem}\label{master2=1}
 If $xy=u_1u_2$ and $x+y=u_3+u_4$, then
\begin{align}
A_n\left(x,y,t\,|\, \a\right)=\sum_{\sigma \in \mathfrak{S}_{n+1}}(u_1u_2)^{{\val}(\sigma)}u_3^{{\lda}(\sigma)}u_4^{{\rdd}(\sigma)}
t^{\lrmaxda(\s)+\rlmaxdd(\s)}{\alpha}^{{\LRmax}(\sigma)+{\RLmax}(\sigma)-2}. 
\end{align}
\end{theorem}
\end{subequations}
We define  the $(\a,t)$-binomial-Eulerian polynomials
$\widetilde{A}_n(x,y,t\,|\,\alpha)$ by
 \begin{subequations}
 \begin{align}
\widetilde{A}_n(x,y,t\,|\,\alpha)
=\sum_{\s\in\PRW_{n+1}}x^{\asc(\s)}y^{\des(\s)}t^{\lrmaxda(\s)+\rlmaxdd(\s)}
\a^{\LRmax(\s)+\RLmax(\s)-2},\label{A-tilde-def}
 \end{align}
 which is equal to $A_{n}(x, y, x, y, 0,1,t\,|\, {\alpha},{\alpha})$ because 
 a permutation $\s\in \S_n$ is an element of $\PRW_n$ if and only if $\fmax(\s)=1$. 
 By  Theorem~\ref{master2} we have 
 \begin{equation}
\widetilde{A}_n(x,y,t\,|\,\alpha)=A_n^{\cyc}( x,y,(x+y)t\,|\,\alpha).\label{cyc-lin1}
 \end{equation}
 Combining \eqref{cyc-lin1} and 
 Theorem~\ref{master2} 
with    $f=0$,  $g=1$, and $\a=\beta$, we obtain   the following $t$-analogue of   Ji-Lin's Theorem 1.5~\cite{JL23}.
 \begin{theorem}\label{master2=0}
If  $xy=u_1u_2$ and $x+y=u_3+u_4$, then
\begin{align}
\widetilde{A}_n(x,y,t\,|\,\alpha)=\sum_{\sigma \in \mathfrak{\PRW}_{n+1}}(u_1u_2)^{{\val}(\s)}u_3^{\lda(\sigma)}u_4^{{\rdd}(\sigma)}
t^{\lrmaxda(\s)+\rlmaxdd(\s)}
{\alpha}^{{\LRmax}(\sigma)+{\RLmax}(\sigma)-2}.
\end{align}
\end{theorem}
 \end{subequations}

  The first four values of $\widetilde{A}_n(x,y,t\,|\,\alpha)$  are as follows:
\begin{align*}
\widetilde{A}_1(x,y,t\,|\,\alpha)&=\a t(x + y), \\
\widetilde{A}_2(x,y,t\,|\,\alpha)&=
\a^{2} t^{2}(x^2+y^{2})+(2 \a^{2} t^{2}+\a) xy,\\
\widetilde{A}_3(x,y,t\,|\,\alpha)&=\a^{3} t^{3}(x^3+y^{3})+\left(3 \a^{3} t^{3}+3 \a^{2} t+\a\right)( x^2y+xy^{2}),\\
\widetilde{A}_4(x,y,t\,|\,\alpha)&=
\a^{4} t^{4}(x^4+ y^{4})+\left(4 \a^{4} t^{4}+6 \a^{3} t^{2}+4 \a^{2} t+\a\right) (x^3y+xy^{3})\\
&\hspace{6cm}+\left(6 \a^{4} t^{4}+12 \a^{3} t^{2}+8 \a^{2} t +3 \a^2+4\a\right) x^2y^{2}.
\end{align*}


\subsection{A symmetric $(\alpha,t)$-Eulerian identity}\label{sec:Eulerian identity}
We need a variant $\theta_1$ of 
Foata's fundamental transformation~\textbf{FFT}, see 
\cite[p.30]{St12}. 
For $\s\in \S_n$,  define the mapping
$\theta_1: \s\mapsto \theta_1(\s)$   as follows:
\begin{itemize}
\item factorize $\s$ as product of disjoint 
cycles with their largest letters at the end of each cycle; 
\item write $\tilde{\s}=C_1\ldots C_r$, where the cycles $C_i's$ are arranged 
 in \textbf{decreasing order of their largest letters} from  left to right;
\item erase the parentheses in $\tilde{\s}$ to obtain 
$\theta_1(\s)$.
\end{itemize}
\begin{lemma}\label{mapping1-properties}
For $\s\in \S_n$, we have
    \begin{itemize}        \item $\exc(\s)=\asc(\theta_1(\s))$;
        \item $\fix(\s)=\rlmaxdd(\theta_1(\s))$;
        \item $\cyc(\s)=\rlmax(\theta_1(\s))$.
    \end{itemize}
\end{lemma}
\begin{example}
    If $\s=4271365$,  then 
$\tilde{\s}=(5\,3\,7)(6)(1\,4)(2)$ and  $\s':=\theta_1(\s)=5376142$. 
We verify that 
\begin{itemize}
    \item 
 $\exc(\s)=|\{1,3\}|=2$, $\fix(\s)=|\{2,6\}|=2$, $\cyc(\s)=4$; 
 \item $\asc(\s')=|\{2,5\}|=2$, $\rlmaxdd(\s')=|\{2,6\}|=2$, $\rlmax(\s')=|\{2,4,6,7\}|=4$.
 \end{itemize}
\end{example}

Define two kinds of $(\a,t)$-\emph{Eulerian numbers}  as follows:
\begin{subequations}
    \begin{equation}\label{(a,t)-Stirling-Eulerian number}
\left\langle \mycom{n}{k}\right\rangle^{\exc}_{\alpha,t}:=\sum_{
\mycom{\s\in \S_n}{\exc(\s)=n-k}}\alpha^{\cyc(\s)}t^{\fix(\s)}\qquad (1\leq k\leq n),
\end{equation}
and 
\begin{equation}\label{linear-stirling-eulerian-number}
\left\langle \mycom{n}{k}\right\rangle_{\alpha,t}^{\asc}:=
\sum_{\mycom{\s\in \S_n}{\asc(\s)=n-k}}\a^{\RLmax(\s)}
t^{\rlmaxdd(\s)} \qquad ( 1\leq k\leq n).
\end{equation}
\end{subequations}
As $\exc+\drop+\fix=n$ over $\S_n$, 
we derive  from \eqref{CS:def1} that 
\begin{align}\label{equ:exc-a-t-id}
A_n^{\cyc}(x,y,ty\,|\, \alpha)&=\sum_{\s\in \S_n} x^{\exc(\s)} y^{\drop(\s)+\fix(\s)}t^{\fix(\s)}\alpha^{\cyc(\s)}
=
\sum_{k=1}^{n}\left\langle \mycom{n}{k}\right\rangle^{\exc}_{\alpha,t} x^{n-k}y^k.
\end{align} 
If $\s\in\S_n$ with $\s_1=n$, then $\lrmaxda(\s)=0$ and $\LRmax(\s)=1$. From~\eqref{A-tilde-def}, we have
\begin{align}\label{equ:asc-a-t-id}
\widehat{A}_n\left(x,y,t\,|\, \a\right)
:=&\sum_{\s\in\PRW_{n+1}\atop \s_1=n+1}x^{\asc(\s)}y^{\des(\s)} t^{\rlmaxdd(\s)}\a^{\RLmax(\s)-1}
\nonumber\\
=&\sum_{\s\in\S_{n}}x^{\asc(\s)}y^{\des(\s)+1} t^{\rlmaxdd(\s)}\a^{\RLmax(\s)}\nonumber\\
=&\sum_{k=1}^{n}
\left\langle \mycom{n}{k}\right\rangle_{\alpha,t}^{\asc}x^{n-k+1}y^{k}.
\end{align}

\begin{theorem}\label{sum-equ:bin-(a,t)} For integers $a, b\geq 0$, we have
\begin{subequations}
    \begin{align}\label{exc=asc}
    \left\langle \mycom{n}{k}\right\rangle_{\alpha,t}:=\left\langle \mycom{n}{k}\right\rangle_{\alpha,t}^{\exc}=\left\langle \mycom{n}{k}\right\rangle_{\alpha,t}^{\asc}, 
\end{align}
and 
\begin{equation}\label{sym1}
\sum_{k\geq 0}(\alpha t)^{a+b-k} \binom{a+b}{k}
\left\langle \mycom{k}{a}\right\rangle_{\alpha,t}=
\sum_{k\geq 0}(\alpha t)^{a+b-k} \binom{a+b}{k}
\left\langle \mycom{k}{b}\right\rangle_{\alpha,t},
\end{equation}
\end{subequations}
where
 $\left\langle \mycom{0}{k}\right\rangle_{\alpha,t}=\left\langle \mycom{k}{0}\right\rangle_{\alpha,t}=\delta_{k,0}$.
\end{theorem}
\begin{proof}
By Lemma~\ref{mapping1-properties}, we have 
\begin{equation}
  (n-\exc,\fix,\cyc)\,\s=(n-\asc,\rlmaxdd,\RLmax)\,\theta_1(\s).
\end{equation}
This is  clearly equivalent to~\eqref{exc=asc}. 
We only prove~\eqref{sym1} for $\left\langle \mycom{n}{k}\right\rangle_{\alpha,t}^{\exc}$ using \eqref{equ:exc-a-t-id}. A similar proof for $\left\langle \mycom{n}{k}\right\rangle_{\alpha,t}^{\asc}$ can be given using \eqref{equ:asc-a-t-id}.

\begin{table}[tp]
  $$
  \begin{array}{c|cccc}
\hbox{$n$}\backslash\hbox{$k$}&1&2&3&4\\
\hline
1& \a t&\\
2& \a &\a^2t^2\\
3& \a&\a+3\a^2t&\a^3t^3&\\
4& \a&4\a+3\a^2+4\a^2t&\a+4\a^2t+6\a^3t^2&\a^4t^4\\
\end{array}
 $$
 \caption{The first values of $\left\langle \mycom{n}{k}\right\rangle_{\alpha,t}$  for $1\le k\le 4$.} \label{fig-value}
 \end{table}
Partition the set of fixed points of each permutation in $\S_n$ in two categories, say \emph{blue ones} and \emph{red ones}, and by~\eqref{equ:exc-a-t-id}, we have 
\begin{align}
A_n^{\cyc}(x,y,t(x+y)\,|\, \alpha)&=\sum_{m=0}^n \binom{n}{m} (\a tx)^{n-m}
A^{\cyc}_m(x, y, ty\,|\,\alpha)\nonumber\\
&=
\sum_{m=0}^n \binom{n}{m} (\a t)^{n-m}\sum_{l=1}^{m}
\left\langle {\mycom{m}{l}}\right\rangle^{\exc}_{\alpha,t}x^{n-l} y^{l},\label{binomial-cyc-stirling-eulerian}
\end{align}
where
 $\left\langle \mycom{0}{k}\right\rangle^{\exc}_{\alpha,t}=\left\langle \mycom{k}{0}\right\rangle^{\exc}_{\alpha,t}=\delta_{k,0}$.
 
 It is easy to see that 
$A_n^{\cyc}(x,y,t(x+y)\,|\,\alpha)$  is symmetric in $x$ and $y$ because the 
involution $\vartheta: \s\mapsto \s^{-1}$ for $\s\in \S_n$ satisfies $(\exc, \drop, \fix)\,\s=(\drop, \exc, \fix)\,\s^{-1}$
(This also follows from the $\gamma$-formula in  Theorem~\ref{gamma:P_n}). 
Extracting the coefficients of $x^ay^b$ and $x^by^a$ with $n=a+b$ in 
Eq.~\eqref{binomial-cyc-stirling-eulerian} we obtain Eq.~\eqref{sym1} for 
$\left\langle \mycom{n}{k}\right\rangle^{\exc}_{\alpha,t}$.


\end{proof}
\begin{remark}
  Identity~\eqref{sym1} is an $(\alpha,t)$-analogue of \eqref{cgk}.
    When $t=1$,   by  complement operation $\s\mapsto \s^c$, 
identity~\eqref{sym1} with $\left\langle \mycom{n}{k}\right\rangle_{\alpha,1}^{\asc}$ yields 
Ji and Lin's Theorem 4.1 in~\cite{JL23}.
\end{remark}


\subsection{$\gamma$-positivity of $(\alpha, t)$-Eulerian polynomials}
In this section we study the $\gamma$-coefficients of  the two kinds of $(\alpha, t)$-Eulerian polynomials $\widetilde{A}_n(x,y,t\,|\,\alpha)$ and $A_n(x,y,t\,|\,\alpha)$.
We show that their $\gamma$-coefficients are polynomials in $\N[\a,t]$  by  providing  various  
combinatorial interpretations. We  also compute their exponential generating functions. 
\begin{lemma}\label{lemma-gamma-3-cyc-linear} For any variable $f$ we have 
    \begin{itemize}
    \item[(i)]
    \begin{subequations}
\label{eq:gamma1}
    \begin{align} 
    A_n&(x,y, 0, x+y, f, 1,  t\,|\,\a, \a)\nonumber\\
    &=\sum_{\s\in\S_{n+1}\atop\da(\s)=0}(xy)^{\asc(\s)}(x+y)^{n-2\asc(\s)}f^{\fmax(\s)-1}t^{\rlmaxdd(\s)}\a^{\LRmax(\s)+\RLmax(\s)-2}\label{eq:gamma1-a}\\
    &=\sum_{\s\in\S_n\atop\cda(\s)=0}(xy)^{\exc(\s)}(x+y)^{n-2\exc(\s)}t^{\fix(\s)}\a^{\cyc(\s)}(f+1)^{\cyc(\s)-\fix(\s)}.\label{eq:gamma1b}
    \end{align}
 \end{subequations}
\item[(ii)]
  \begin{subequations} \label{eq:gamma2}
     \begin{align} 
        A_n&(x,y, x+y, 0, f, 1,  t\,|\,\a, \a)\nonumber\\
        &=\sum_{\s\in\S_{n+1}\atop\dd(\s)=0}(xy)^{\des(\s)}(x+y)^{n-2\des(\s)}f^{\fmax(\s)-1}t^{\lrmaxda(\s)}\a^{\LRmax(\s)+\RLmax(\s)-2}\label{eq:gamma2a}\\  
      &=\sum_{\s\in\S_n\atop\cdd(\s)=0}(xy)^{\drop(\s)}(x+y)^{n-2\drop(\s)}t^{\fix(\s)}\a^{\cyc(\s)}(f+1)^{\cyc(\s)-\fix(\s)}.\label{eq:gamma2b}
    \end{align}  
  \end{subequations}  
  \item[(iii)]
  \begin{subequations}\label{eq:gamma3}
      \begin{align}
   & A_n\biggl(1,xy, \frac{x+y}{2},
    \frac{x+y}{2}, f, 1,
    t\,|\,\a,\a \biggr)\nonumber\\
        &=\sum_{\s\in\S_{n+1}}(xy)^{\val(\s)}\bigg(\frac{x+y}{2}\bigg)^{n-2\val(\s)}f^{\fmax(\s)-1}t^{\lrmaxda(\s)+\rlmaxdd(\s)}\a^{\LRmax(\s)+\RLmax(\s)-2}\label{eq:gamma3a}\\  
      &=\sum_{\s\in\S_n}(xy)^{\cpk(\s)}\bigg(\frac{x+y}{2}\bigg)^{n-2\cpk(\s)}(2t)^{\fix(\s)}\a^{\cyc(\s)}(f+1)^{\cyc(\s)-\fix(\s)}.\label{eq:gamma3b}
    \end{align}
  \end{subequations} 
   \end{itemize}
\end{lemma}
\begin{proof}
First of all, by definition (see Section~1), it is easy to verify the following identities: 
\begin{subequations}\label{relation:cyc-linear-statistics}
\begin{align}
\val=\pk-1,\, \asc=\val+\lda,\, \des=\val  +\rdd;\\
\cval=\cpk,\,\exc=\cval+\cda,\, \drop=\cpk+\cdd;\\
\fix+\cdd+\cda=n-\cpk-\cval=n-2\cpk.
 \end{align}  
\end{subequations}
When $\a=\beta$ and $g=1$,  Theorem~\ref{master2} with $xy=u_1u_2$ and $x+y=u_3+u_4$ becomes
     \begin{equation}\label{equ:f-1-t}
             A_n(u_1,u_2, u_3, u_4, f, 1,  t\,|\,\a, \a)=A_n^{\cyc}\left( x,y, \frac{u_3+u_4}{f+1}\,t\,|\,  \a f+\a\right).
        \end{equation}
\begin{itemize}
    \item[(i)]       Let $\uu=(x,y,0,x+y)$ in Eq.~\eqref{equ:f-1-t}, then the
    left-hand side becomes $A_n(x,y, 0, x+y, f, 1,  t\,|\,\a, \a)$, which is  equal to  \eqref{eq:gamma1-a}  by definition~\eqref{refine-JI2}~and~\eqref{relation:cyc-linear-statistics}, while the right-hand side of Eq.~\eqref{equ:f-1-t} becomes~\eqref{eq:gamma1b} by Theorem~\ref{master1}.
    
     \item[(ii)] Let $\uu=(x,y,x+y,0)$ in Eq.~\eqref{equ:f-1-t}, similarly, we prove Eq.~\eqref{eq:gamma2a} and \eqref{eq:gamma2b}. 
    
    \item[(iii)] Let  $\uu=(1,xy, \frac{x+y}{2},
    \frac{x+y}{2})$ in Eq.~\eqref{equ:f-1-t}, then the left-hand side becomes $$
    A_n\left(1,xy, \frac{x+y}{2},
    \frac{x+y}{2}, f, 1,
    t\,|\,\a,\a\right),$$ which is equal to 
    \eqref{eq:gamma3a} by definition~\eqref{refine-JI2} and~\eqref{relation:cyc-linear-statistics}, while the right-hand side of \eqref{equ:f-1-t} becomes~\eqref{eq:gamma3b} by Theorem~\ref{master1}.
    
\end{itemize}


\end{proof}

For convenience of description of the $\g$-coefficients, we shall define three types of subsets of $\S_n$ with respect to cyclic statistics and linear statistics of permutations, respectively.  

 The cycle type subsets of $\S_n$:
\begin{subequations}\label{cycle-subset-cda-cdd-cpk}
 \begin{align}
&  \S^{\cda=0}_{n, \exc=j}:=\{\s\in \S_n:\cda(\s)=0\,\, {\rm and}\,\,  \exc(\s)=j\};\label{cda=0}\\  
& \S^{\cdd=0}_{n, \drop=j}:=\{\s\in \S_n: \cdd(\s)=0\,\, {\rm and}\,\, \drop(\s)=j\};\label{cdd=0}\\
&\S_{n}^{\cpk=j}:=\{\s\in\S_n: \cpk(\s)=j\}.\qquad\label{cpk}
\end{align}
 \end{subequations}
 
 The linear counterpart  subsets of $\S_n$:
 \begin{subequations}\label{linear-subset-da-dd-W}
\begin{align}
&   \S^{\da=0}_{n,\asc=j}:=\{\s\in \S_n: \da(\s)=0\,\, {\rm and}\,\,  \asc(\s)=j\};\label{da=0}\\
& \S^{\dd=0}_{n,\des=j}:=\{\s\in \S_n: \dd(\s)=0\,\, {\rm and}\,\, \des(\s)=j\};\label{dd=0}\\
&\S^{\pk=j}_{n}:=\{\s\in\S_n: \pk(\s)=j\};\label{peak-PRW}
\end{align}
 \end{subequations}
 and the subsets of $\PRW_n$:
 \begin{subequations}\label{linear-subset-PRW-da-dd-W}
\begin{align}
&   \PRW^{\da=0}_{n,\asc=j}:=\{\s\in \PRW_n: \da(\s)=0\,\, {\rm and}\,\,  \asc(\s)=j\};\label{da=0-PRW}\\
& \PRW^{\dd=0}_{n,\des=j}:=\{\s\in \PRW_n: \dd(\s)=0\,\, {\rm and}\,\, \des(\s)=j\};\label{dd=0-PRW}\\
&\PRW^{\pk=j}_{n}:=\{\s\in\PRW_n: \pk(\s)=j\}.\label{W}
\end{align}
 \end{subequations}

From Lemma~\ref{lemma-gamma-3-cyc-linear}, we derive the following combinatorial interpretations of the coefficients in the $\gamma$-expansion of 
$\widetilde{A}_n(x,y,t\,|\,\alpha)=A_n^{\cyc}( x,y,(x+y)t\,|\,\alpha)$ based on linear and cyclic statistics.

\begin{theorem}\label{gamma:P_n}
For $n\ge 1$, we have
\begin{align}
\widetilde{A}_n(x,y,t\,|\,\alpha)
&=\sum_{j=0}^{\lfloor \frac{n}{2}\rfloor}\widetilde{\gamma}_{n,j}(\a,t)(xy)^j(x+y)^{n-2j},\label{equ:gamma:P_n}
\end{align}
where 
\begin{subequations}\label{gamma1}
\begin{align}
\widetilde{\gamma}_{n,j}(\a,t)&=\sum_{\PRW^{\da=0}_{n+1,\asc=j}}\a^{\RLmax(\s)-1}t^{\rlmaxdd(\s)}\label{gamma1a}\\
&=\sum_{\sigma\in  \S^{\cda=0}_{n, \exc=j}}\alpha^{\cyc\,(\sigma)}t^{\fix(\sigma)};\label{gamma1b}
\end{align}
\end{subequations}
\begin{subequations}\label{gamma2}
\begin{align}
\widetilde{\gamma}_{n,j}(\a,t)&=\sum_{\PRW^{\dd=0}_{n+1,\des=j}}\a^{\LRmax(\s)+\RLmax(\s)-2}t^{\lrmaxda(\s)}\label{gamma2a}\\
&=\sum_{\sigma\in  \S^{\cdd=0}_{n, \drop=j}}\alpha^{\cyc\,(\sigma)}t^{\fix(\sigma)};\label{gamma2b}
\end{align}
\end{subequations}
\begin{subequations}\label{gamma3}
\begin{align}
\widetilde{\gamma}_{n,j}(\a,t)&=2^{2j-n}\sum_{\PRW_{n+1}^{\pk=j}}\a^{\LRmax(\s)+\RLmax(\s)-2}t^{\lrmaxda(\s)+\rlmaxdd(\s)}\label{gamma3a}\\
&=2^{2j-n}\sum_{\s\in\S_{n}^{\cpk=j}}\a^{\cyc(\s)}(2t)^{\fix(\s)}.\label{gamma3b}
\end{align}
\end{subequations}
\end{theorem}
\begin{proof}
Let $u_1u_2=xy$ and $u_3+u_4=x+y$,  combining \eqref{cyc-lin1} and \eqref{equ:f-1-t}, we obtain
\begin{align}\label{key1}
\widetilde{A}_n(x,y,t\,|\,\alpha)= A_n^{\cyc}( x,y,(x+y)t\,|\,\alpha)=A_n(u_1,u_2, u_3, u_4, 0, 1,  t\,|\,\a, \a).
\end{align} 
Combining \eqref{key1} with $u_3=0$ (resp. $u_4=0$) and 
\eqref{eq:gamma1} (resp. \eqref{eq:gamma2}) with $f=0$,  we obtain the twins interpretations~\eqref{gamma1}  (resp. \eqref{gamma2}). Note that we obtain 
\eqref{gamma1a} by using the fact that  
if  $\s\in\PRW_n$ and  $\da(\s)=0$, then  $\LRmax(\s)=1$.

Similarly,  let $u_1=1$, $u_2=xy$ and $u_3=u_4=\frac{x+y}{2}$ in \eqref{key1},   we obtain
\begin{align}\label{key2}
  \widetilde{A}_n(x,y,t\,|\,\alpha)= A_n^{\cyc}( x,y,(x+y)t\,|\,\alpha)=A_n\biggl(1,xy, \frac{x+y}{2},
    \frac{x+y}{2}, 0, 1,
    t\,|\,\a,\a \biggr).
\end{align} 
Combining \eqref{key2} and \eqref{eq:gamma3}, we obtain the twins interpretations~\eqref{gamma3}.
\end{proof}
\begin{remark}
\begin{enumerate}
\item
 Previously, 
Shin-Zeng~\cite[Theorem~11]{SZ12} proved the $t=0$ case of \eqref{gamma1b}, i.e.,
\begin{equation}\label{sz12}
A_{n}^{\cyc}( x,y,0\,|\,\alpha)=
\sum_{\s\in \DD^*_n}
\alpha^{\cyc(\s)} (xy)^{\exc(\s)} (x+y)^{n-2\,\exc(\s)},
\end{equation}
where $\DD^*_{n}=\{\sigma\in \S_{n}: \cda(\s)=0,\; \fix(\s)=0\}$.

In fact, it is easy to  derive  the general case \eqref{gamma1b} from \eqref{sz12} 
 as in the following. Counting the permutations in $\S_n$ by the number of fixed points and 
 plugging \eqref{sz12}, we obtain 
\begin{align}
A_n^{\cyc}( x,y,t(x+y)\,|\,\alpha)&=\sum_{k=0}^n \binom{n}{k}
(\a t(x+y))^k A_{n-k}^{\cyc}( x,y,0\,|\,\alpha)\nonumber\\
&=\sum_{k=0}^n \binom{n}{k}
(\a t)^k \sum_{\s\in \DD^*_{n-k}}\alpha^{\cyc(\s)} (xy)^{\exc(\s)}(x+y)^{n-2\exc(\s)}\nonumber\\
&= \sum_{j=0}^{\lfloor n/2\rfloor}\biggl(\sum_{\s\in\S^{\cda=0}_{n,\exc=j}}\alpha^{\cyc(\s)}t^{\fix(\s)} \biggr)
(xy)^j (x+y)^{n-2j}.\label{SZ-bis}
\end{align}
\item An equivalent form of  \eqref{SZ-bis}
appeared in  \cite[Theorem 3.4]{M2Y224} and \cite[Theorem~7]{PS24}. One can also derive this result  from \cite[Theorem~6]{CJZ20}.
\item When $t=1$, by  complement operation $\s\mapsto \s^c$, the three combinatorial interpretations~\eqref{gamma1a}, \eqref{gamma2a} and \eqref{gamma3a} yield  the  three combinatorial interpretations of the $\gamma$-coefficients in~\cite[Theorem 2.1]{JL23}.

\end{enumerate}
\end{remark}

The $\gamma$-expansions of 
the first four  polynomials $\widetilde{A}_n(x,y,t\,|\,\alpha)$ are as follows:
\begin{align*}
&\widetilde{A}_1(x,y,t\,|\,\alpha)=\a t (x+y),\\
&\widetilde{A}_2(x,y,t\,|\,\alpha)=\a^2t^2(x+y)^2+\a xy,\\
&\widetilde{A}_3(x,y,t\,|\,\alpha)=\a^3t^3(x+y)^3+(\a+3\a^2)t xy(x+y),\\
&\widetilde{A}_4(x,y,t\,|\,\alpha)=\a^4t^4(x+y)^4+(6\a^3+4\a^2+\a)t^2xy(x+y)^2+(3\a^2+2\a)(xy)^2.
\end{align*}


From Lemma~\ref{lemma-gamma-3-cyc-linear},  we obtain the following  combinatorial interpretations for the $\gamma$-coefficients
 of $A_n(x,y,t\,|\, \a)=A_n^{\cyc}\left(x,y,\frac{x+y}{2}t\,|\, 2\a\right)$ based on linear and cyclic statistics.
\begin{theorem}
    \label{coro: gamma-expansion-A_n(x,y|a)-cyc-version}
For $n\ge 0$ and $0\leq j\leq \lfloor{n/2}\rfloor$, we have

\begin{equation}\label{gamma:A}
A_n(x,y,t\,|\, \a)=\sum_{j=0}^{\lfloor n/2\rfloor}\gamma_{n,j}(\a,t)(xy)^j(x+y)^{n-2j},
\end{equation}
where 
\begin{subequations}\label{gamma-linear1}
 \begin{align}
\gamma_{n,j}(\a,t)&=\sum_{\s\in\S^{\da=0}_{n+1,\asc=j}}\a^{\LRmax(\s)+\RLmax(\s)-2}t^{\rlmaxdd(\s)}\label{gamma-linear1a}\\
&=\sum_{\s\in\S^{\cda=0}_{n,\exc=j}}2^{\cyc(\s)-\fix(\s)}\a^{\cyc(\s)}t^{\fix(\s)};\label{gamma-linear1b}
\end{align}
\end{subequations}
\begin{subequations}\label{gamma-linear2}
\begin{align}
\gamma_{n,j}(\a,t)&=\sum_{\s\in\S^{\dd=0}_{n+1,\des=j}}\a^{\LRmax(\s)+\RLmax(\s)-2}t^{\lrmaxda(\s)}\label{gamma-linear2a}\\
&=\sum_{\s\in\S^{\cdd=0}_{n,\drop=j}}2^{\cyc(\s)-\fix(\s)}\a^{\cyc(\s)}t^{\fix(\s)};\label{gamma-linear2b}
\end{align}
\end{subequations}
\begin{subequations}\label{gamma-linear3}
\begin{align}
\gamma_{n,j}(\a,t)&=2^{2j-n}\sum_{\s\in\S^{\pk=j+1}_{n+1}}\a^{\LRmax(\s)+\RLmax(\s)-2}t^{\lrmaxda(\s)+\rlmaxdd(\s)}\label{gamma-linear3a}\\
&=2^{2j-n}\sum_{\s\in\S^{\cpk=j}_{n}}(2\a)^{\cyc(\s)}t^{\fix(\s)}.\label{gamma-linear3b}
\end{align}
\end{subequations}
\end{theorem}
 \begin{proof}
Let $f=1$, $u_1u_2=xy$ and $u_3+u_4=x+y$
from \eqref{An=Ancyc} and \eqref{equ:f-1-t}   we derive 
\begin{equation}\label{key3}
A_n(x,y,t\,|\,\alpha)=A_n^{\cyc}\left(x,y,\frac{x+y}{2}t\,|\, 2\a\right)=A_n(u_1,u_2, u_3, u_4, 1, 1,  t\,|\,\a, \a).
\end{equation} 

Combining \eqref{key3} with $u_3=0$ (resp. $u_4=0$) and 
\eqref{eq:gamma1} (resp. \eqref{eq:gamma2}) with $f=1$,  we obtain the twins interpretations~\eqref{gamma-linear1}  (resp. \eqref{gamma-linear2}).

Similarly,  let $u_1=1$, $u_2=xy$ and $u_3=u_4=\frac{x+y}{2}$ in \eqref{key3},  we obtain
\begin{equation}\label{key4}
A_n(x,y,t\,|\,\alpha)=A_n^{\cyc}\left(x,y,\frac{x+y}{2}t\,|\, 2\a\right)=A_n\biggl(1,xy, \frac{x+y}{2},
    \frac{x+y}{2}, 1, 1,
    t\,|\,\a,\a \biggr).
\end{equation} 
Combining \eqref{key4} and \eqref{eq:gamma3}, we obtain the twins interpretations~\eqref{gamma-linear3}.
\end{proof}

\begin{table}[t]
  $$
  \begin{array}{c|cccc}
\hbox{$n$}\backslash\hbox{$j$}&0&1&2\\
\hline
0& 1&\\
1& \a t&\\
2& \a^2t^2&2\a&&\\
3& \a^3t^3&2\a+6\a^2 t&&\\
4& \a^4t^4&2\a+8\a^2t+12\a^3t^2&4\a+12\a^2&\\
5& \a^5t^5&2\a+10\a^2t+20\a^3t^2+20\a^4t^3&16\a+40\a^2+20\a^2t+60\a^3t
\end{array}
 $$
 \caption{The first values of $\gamma_{n,j}(\a,t)$ for $0\le 2j<n\le 5$.} \label{table2}
 \end{table}
The first values of $\gamma_{n,j}(\a,t)$ are given in Table~\ref{table2}. 
 The following result provides   a $t$-analogue of Ji's Theorem~1.9~\cite{Ji23}.
\begin{theorem}For $n\ge 0$, we have
\begin{subequations}
\begin{equation}
A_n\bigg(x,y,t\,|\,\frac{\a+\be}{2}\bigg)
= \sum_{j=0}^{\lfloor n/2\rfloor}\gamma_{n,j}\left(\frac{\a+\be}{2}, t\right)(xy)^j(x+y)^{n-2j},
\end{equation}
where
\begin{align}
\gamma_{n,j}\left(\frac{\a+\be}{2}, t\right)
=2^{2j-n}\sum_{\s\in\S_{n+1}^{\pk=j+1}}\a^{\LRmax(\s)-1}\be^{\RLmax(\s)-1}t^{\lrmaxda(\s)+\rlmaxdd(\s)}.\label{gamma-peak-1}
\end{align}
\end{subequations}
\end{theorem}
\begin{proof}
By~\eqref{An=Ancyc}, then
\begin{equation}\label{A-alpha+beta}
    A_n\left(x,y,t\,|\,\frac{\a+\be}{2}\right)=A_n^{\cyc}\biggl(x,y, \frac{x+y}{2}\,t\,|\,\a+\be\biggr).
\end{equation}
Let $\uu=(1,xy,\frac{x+y}{2},\frac{x+y}{2})$, by  Theorem~\ref{master2} and definition~\eqref{refine-JI2}, we have 
\begin{align}\label{eq:def-S_n+1}
&A_n^{\cyc}\biggl(x,y, \frac{x+y}{2}\,t\,|\,\a+\be\biggr)\nonumber\\
    &=A_n\biggl(1,xy, \frac{x+y}{2},
    \frac{x+y}{2},1, 1,
    t\,|\,\a,\be \biggr)\nonumber\\
    &=\sum_{\s\in\S_{n+1}}(xy)^{\val(\s)}(x+y)^{n-2\val(\s)}t^{\lrmaxda(\s)+\rlmaxdd(\s)}\a^{\LRmax(\s)-1}\be^{\RLmax(\s)-1}2^{2\val(\s)-n}.
\end{align}
As $\val=\pk-1$, combining~\eqref{A-alpha+beta} and~\eqref{eq:def-S_n+1} yields \eqref{gamma-peak-1}. 
\end{proof}

\begin{remark}
Combining \eqref{gamma-linear3a} and \eqref{gamma-peak-1} we obtain
\begin{align}\label{t-DLP}
\sum_{\s\in\S_{n+1}^{\pk=j+1}}&\a^{\LRmax(\s)-1}\be^{\RLmax(\s)-1}t^{\lrmaxda(\s)+\rlmaxdd(\s)}\nonumber\\
&=\sum_{\s\in\S_{n+1}^{\pk=j+1}}\left(\frac{\a+\be}{2}\right)^{\LRmax(\s)+\RLmax(\s)-2}t^{\lrmaxda(\s)+\rlmaxdd(\s)}.
\end{align}
 When $t=1$, Dong, Lin and Pan~\cite[Theorem 1.5]{DLP24} recently 
 proved  \eqref{t-DLP} by a group action.
\end{remark}

Define the generating polynomials of the two kinds of 
$\g$-vectors in  \eqref{equ:gamma:P_n} and~\eqref{gamma:A}
\begin{subequations}
\begin{align}
G^{\cyc}_n(x,t,\alpha):=&
\sum_{j=0}^{\lfloor \frac{n}{2}\rfloor}\widetilde{\gamma}_{n,j}(\alpha, t)x^j=\sum_{\s\in \S^{\cda=0}_n}x^{\exc(\s)}t^{\fix(\sigma)}\alpha^{\cyc\,(\sigma)},\\
G_n(x,t,\alpha):=&
\sum_{j=0}^{\lfloor \frac{n}{2}\rfloor}\gamma_{n,j}(\alpha, t)x^j=\sum_{\s\in \S^{\da=0}_{n+1}}x^{\asc(\s)}\a^{\LRmax(\s)+\RLmax(\s)-2}t^{\rlmaxdd(\s)},
\end{align}
where
\begin{align}\label{subset: cda-da}
\S^{\cda=0}_n:=\{\s\in\S_n: \cda(\s)=0\}\quad {\rm and}\quad \S^{\da=0}_n:=\{\s\in\S_n: \da(\s)=0\}. 
\end{align}

\begin{theorem}\label{Thm: egf of gamma coefficients}
 Let $u=\sqrt{1-4x}$. We have
\begin{align}
1+\sum_{n\ge 1}G^{\cyc}_n(x,t,\alpha)\frac{z^n}{n!}
&=\biggl(\frac{u\,e^{(t-\frac{1}{2})z}}{u\cosh(uz/2)-
\sinh(uz/2) }\biggr)^{\alpha},\label{equ:gen-Gamma-coe}\\
1+\sum_{n\ge 1}G_n(x,t,\alpha)\frac{z^n}{n!}
&=\biggl(\frac{u\,e^{\frac{1}{2}(t-1)z}}{u\cosh(uz/2)-
\sinh(uz/2) }\biggr)^{2\alpha}.\label{equ2:gen-Gamma-coe}
\end{align} 
\end{theorem}
\end{subequations}
\begin{proof}
With $y=1$ and $s=\frac{x}{(1+x)^2}$,  the two $\g$-expansions \eqref{equ:gamma:P_n} and~\eqref{gamma:A} can be written, respectively, as follows, 
\begin{subequations}
\begin{align}
&A^{\cyc}_n(x,1,t(x+1)\,|\,\a)=(1+x)^nG^{\cyc}_n(s,t,\alpha),\label{equ:gamma-A^cyc}\\
&A_n(x,1,t\,|\,\a)=(1+x)^nG_n(s,t,\alpha).\label{equ:gamma-A^lin}
\end{align}
 Plugging~\eqref{equ:gamma-A^cyc} in~\eqref{equ:gen-CS-cyc} yields 
\begin{align}\label{equ:gen-Gamma-coe-exp}
1+\sum_{n\ge 1}G^{\cyc}_n(s,t,\alpha)\frac{z^n}{n!}
=\biggl(\frac{(x-1)e^{tz}}{xe^{z/(1+x)}-e^{xz/(1+x)}}\biggr)^{\alpha}, 
\end{align}
After a little computation, the right-hand sides of~\eqref{equ:gen-Gamma-coe} and~\eqref{equ:gen-Gamma-coe-exp} agree  under substitution $u=\frac{1-x}{1+x}=\sqrt{1-4s}$. 

From~\eqref{expgen:P_n-2} with $u_1=u_3=x$,  $u_2=u_4=f=g=1$,  we derive
\begin{equation}\label{gen: A(x,1,t,a)}
1+\sum_{n\ge 1}A_n(x,1,t\,|\,\a)\frac{z^n}{n!}=e^{\a (x+1)tz}\left(\frac{x-1}{xe^{z}-e^{xz}}\right)^{2\a}.
\end{equation}
\end{subequations}
Formula~\eqref{equ2:gen-Gamma-coe} can be proved similarly by plugging~\eqref{equ:gamma-A^lin} in~\eqref{gen: A(x,1,t,a)}. 
\end{proof}
\begin{remark}
Formula \eqref{equ:gen-Gamma-coe} appeared in \cite[Eq.~(13)]{M2Y224} with a different proof, and \eqref{equ2:gen-Gamma-coe} is a $t$-analogue of Carlitz and Scoville's formula~\cite[Eq.~(4.10)]{CS74}. 
\end{remark}

\subsection{$\gamma$-vetcor of $(\a,t)$-Eulerian polynomials  and cycle André permutations}
For $n\ge 0$ and $0\le j\le \lfloor n/2\rfloor$, let 
\begin{equation}
    d_{n,j}(\a,t)=\gamma_{n,j}(\a,t)/2^j. 
\end{equation}
Then, Eq.~\eqref{gamma:A} reads
\begin{equation}\label{gamma:A/2}
A_n(x,y,t\,|\, \a)=\sum_{j=0}^{\lfloor n/2\rfloor}2^jd_{n,j}(\a,t)(xy)^j(x+y)^{n-2j}.
\end{equation}
From Table~\ref{table2} we derive 
the first values of $d_{n,j}(\a,t)$ in Table~\ref{table3}. A glance at these values foretells that the polynomials $d_{n,j}(\a,t)$ have  nonnegative integeral coefficients.
Indeed, from Theorem~\ref{con-frac-A}, 
we derive the  continued fraction
    \begin{subequations}   
    \begin{equation}
  \sum_{n=0}^\infty \sum_{j=0}^{\lfloor n/2\rfloor} d_{n,j}(\a,t) x^j z^n
 = \cfrac{1}{1-b_0 z-\cfrac{\lambda_1z^2}{1-b_1z-\cfrac{\lambda_2 z^2}{1-b_2z-\cdots}}},
\end{equation}
where
\begin{align}
b_k&=k+\a t,\\
\lambda_{k+1}&=\binom{k+1}{2}x+\a (k+1) x\quad (k\geq 0).
\end{align}
\end{subequations}
It follows   that the polynomials $d_{n,j}(\a,t)$ are elements  in $\N[\a,t]$.   The aim of this section is to provide three combinatorial interpretations in terms of cycle André and Andr\'e permutations.

\begin{table}[t]
  $$
  \begin{array}{c|cccc}
\hbox{$n$}\backslash\hbox{$j$}&0&1&2\\
\hline
0& 1&\\
1& \a t&\\
2& \a^2t^2&\a&&\\
3& \a^3t^3&\a+3\a^2 t&&\\
4& \a^4t^4&\a+4\a^2t+6\a^3t^2&\a+3\a^2&\\
5& \a^5t^5&\a+5\a^2t+10\a^3t^2+10\a^4t^3&4\a+10\a^2+5\a^2t+15\a^3t
\end{array}
 $$
 \caption{The values of $d_{n,j}(\a,t)$ for $0\le 2j<n\le 5$.}\label{table3}
 \end{table}

Note that when $x=\a=t=1$ Eq.\eqref{gamma:A/2} is equivalent to the $\gamma$-expansion of Eulerian polynomials~\eqref{euler-gamma} and $d_{n,j}(1,1)$ is the number of Andr\'e permutations of $[n+1]$ with $j$ descents.

There are many equivalent ways to define an André permutation. Following 
Foata-Strehl~\cite{FS74} and 
Hetyei-Reiner~\cite{HR98}, we define André permutations in terms of \emph{$x$-factorization}. 
\begin{definition}
Let $\s=\s_1\s_2\dots\s_n\in\S_n$ and let $x=\s_i$ for $i\in [n]$. The $x$-factorization of $\s$ is given by $\s=u\,\lambda(x)\,x\,\rho(x)\,v$, where
\begin{itemize}
\item[(1)] $\lambda(x)=\s_j\cdots\s_{i-1}$ with $\s_l>x$ for $j\le l\le i-1$, $u=\s_1\cdots\s_{j-1}$, and $\s_{j-1}<x$.
\item[(2)] $\rho(x)=\s_{i+1}\cdots\s_{k}$ with $\s_l>x$ for $i+1\le l\le k$, $u=\s_{k+1}\cdots\s_{n}$, and $\s_{k+1}<x$.
\end{itemize}
Here,  any of words $u,\lambda(x),\rho(x),v$ may be the empty. 
\end{definition}

\begin{example} If $\s=8\,10\,1\,4\,5\,11\,3\,6\,12\,2\,7\,9\in \S_{12}$, then, the $3$-factorization of $\s$ is 
$$\underbrace{8\,10\,1}_{u}\,\underbrace{4\,5\,11}_{\lambda(3)}\,3\,\underbrace{6\,12}_{\rho(3)}\,\underbrace{2\,7\,9}_{v}.$$
\end{example}

\begin{definition}[\cite{HR98}]\label{def:x-factorization-André1-2}
Let $\s=\s_1\cdots\s_n\in\S_n$. 
Say that $\s$ is an \textbf{André permutation of the first kind (resp. second kind)} 
if $\s$ has no double descents, i.e., $\s_{i-1}>\s_{i}>\s_{i+1}$, and
each $x$-factorisation $u\,\lambda(x)\,x\,\rho(x)\,v$ of $\s$ has property
\begin{itemize}
\item $\lambda(x)=\emptyset$ if $\rho(x)=\emptyset$,
\item $\max(\lambda(x))<\max(\rho(x))$  $($resp. $\min(\rho(x))<\min(\lambda(x))$$)$  if $\rho(x))\neq\emptyset$ and $\lambda(x)\neq\emptyset$.
\end{itemize}
\end{definition}

Let $\A_n^1$ (resp. $\A_n^2$) be the set of André permutations of the first (resp. second) kind  in $\S_n$.  For $n\in [5]$ these permutations are given in  Tables~\ref{example André permutations1} and~\ref{example André permutations2}. 
It is known~\cite{FS70} that the cardinality of 
$\A_n^1$ (resp. $\A_n^2$) is the  \emph{Euler number} $E_{n}$, which is the $n$th coefficient in  the Taylor expansion of $\sec(z)+\tan(z)$, namely,
\begin{equation}
    \sum_{n\ge 0}E_n\frac{z^n}{n!}=\sec(z)+\tan(z).
\end{equation}
The sequence  $E_n$'s  starts with  $1,\, 1,\,2,\,5,\,16,\,61,\,271,\dots$, see~\cite[A138265]{OEIS}.

Let $\s=\s_1\dots\s_n\in\S_n$. A right-to-left minimum of $\s$ is an element $\s_i$ such that $\s_j<\s_i$ if $j>i$. Let $\rlmin(\s)$ be the number of right-to-left minima of $\s$. An index $i\in[n]$ is a right-to-left-minimum-da if it is a double ascent and $\s_i$ is a right-to-left minimum. Let $\rlminda(\s)$ be the number of  right-to-left-minimum-das of $\s$.

%
%

Let $A:=\{a_1,\dots, a_k\}$ be a set of $k$ positive integers. 
Let $C=(a_1,\dots, a_k)$ be a cycle (cyclic permutation) of $A$  with $a_1=\min\{a_1,\dots, a_k\}$. Then,  cycle  $C$ is called an
\textbf{\emph{André cycle}} if the word 
 $a_2\ldots a_k$ is an André permutation of the first kind.
 We say that a  permutation is a \textbf{cycle André permutation}
if it is a product of disjoint André cycles.  Let $\Web_{n}$ be the set of cycle André   permutations  of $[n]$. For example, $(1,5,6)(2,7,4,8)(3)$ is a cycle André   permutation, but $(1,2,5,3)(4)$ is not because $253\sim 132$, which is not an Andr\'e permutation of the first kind.
\begin{remark}
  Hwang et al.~\cite{HJO23} used   cycle André permutations  to characterise 
  the so-called \emph{Web permutations}.   
\end{remark}


\begin{table}[t]\label{AndreI}
\begin{tabular}{ll} 
        \hline
       $n=1$: &  1; \quad $n=2$:\quad 12;\hfil\\
       $n=3$: &  123, 213; \\
       $n=4$:&  1234, 1324, 2314, 2134, 3124; \\
      $n=5$:&12345, 12435, 13425, 23415, 13245, 14235, 34125, 24135, \\
\hphantom{$n=5$:}&23145, 21345, 41235, 31245, 21435, 32415, 41325, 31425. \\
        \hline
    \end{tabular} 
\caption{Andr\'e permutations of the first  kind of length $\leq 5$.}\label{example André permutations1}
\end{table}
\begin{table}[t]\label{AndreII}
\begin{tabular}{ll} 
 \hline
$n=1$:& 1;\quad $n=2$:\quad 12;\\
$n=3$:& 123, 312;\\
$n=4$: &1234, 1423, 3412, 4123, 3124;\\
$n=5$:& 12345, 12534, 14523, 34512, 15234, 14235, 34125, 
45123,\\
\hphantom{$n=5$:}&35124, 51234, 41235, 31245, 51423, 53412, 41523, 31524.\\
 \hline
    \end{tabular} 
\caption{Andr\'e permutations of the second kind of length $\leq 5$.}\label{example André permutations2}
\end{table}
\begin{theorem}\label{main-theorem1}
For $n\ge 0$ and $0\le j\le \lfloor n/2\rfloor$, we have 
\begin{subequations}
\begin{align}
d_{n,j}(\a,t)=&\sum_{\mycom{\s\in\Web_{n}}{\drop(\s)=j}}t^{\fix(\s)}\a^{\cyc(\s)},\label{combinatorial interpretation of d-web}\\
d_{n,j}(\a,t)=&\sum_{\mycom{\s\in \A_{n+1}^1}{\des(\s)=j}}t^{\rlminda(\s)}\a^{\rlmin(\s)-1}\label{combinatorial interpretation of d-andre},\\
d_{n,j}(\a,t)=&\sum_{\mycom{\s\in \A_{n+1}^2}{\des(\s)=j}}t^{\rlminda(\s)}\a^{\rlmin(\s)-1}.\label{combinatorial interpretation of d-andre2}
\end{align}
\end{subequations}
\end{theorem}

To prove~\eqref{combinatorial interpretation of d-web}, we show that  the exponential generating functions of both sides are equal;
we then derive~\eqref{combinatorial interpretation of d-andre} from~\eqref{combinatorial interpretation of d-web}  by using a bijection from $\Web_n$ to $\A_{n+1}^{1}$, and we derive ~\eqref{combinatorial interpretation of d-andre2} from \eqref{combinatorial interpretation of d-andre} by constructing  a bijection from  $\A_{n+1}^1$ to $\A_{n+1}^2$ via   André trees.

\section{Cyclic valley-hopping: 
Proof of Theorem~\ref{master1} and Theorem~\ref{master2}}


\subsection{Cyclic valley-hopping}\label{cyclic valley-hopping}

Fix a letter $x\in[n]$ and a permutation $\s\in\S_n$, the \emph{x-factorization} of $\s$ is defined as the concatenation $\s=w_1w_2\,x\,w_4w_5$, where $w_2$ (resp. $w_4$) is the maximal contiguous subword immediately to the left (resp. right) of $x$ whose letters are all smaller than $x$. For example, if $\s=934278516$ and $x=7$, then $w_1=9$, $w_2=342$, $w_4=\emptyset$ and $w_5=8516$.  Define the involution $\xi_x(\s):\S_n\to\S_n$ by
$$
\xi_x(\s):=
\begin{cases}
w_1w_4\,x\,w_2w_5,& \text{if $x$ is a double ascent or a double descent of $\s$};\\
\s,&\text{if $x$ is a valley  or a peak of $\s$}.\\
\end{cases}
$$
Here, we use conventions $\s(0)=\s(n+1)=\infty$.
It is known that  the involutions $\xi_x$ and $\xi_y$ commute with each other, see~\cite{Br08}. Hence,
for any subset $S\subseteq [n]$, we can define an operator $\xi_S:=\prod_{x\in S}\xi_x$. As shown by Br\"anden~\cite{Br08}, 
the involutions $\xi_S$ ($S\subseteq [n]$) define a $\Z_2^n$-action on $\S_n$, which 
 are called \emph{modified Foata and Strehl actions} or  \emph{MFS}-action, see
 ~\cite{FS74}.
 A geometric version of this action was 
given  by Shapiro et al.~\cite{SWG83}.

For convenience,  we describe another variant  $\theta_2$ of 
Foata's fundamental transformation~\textbf{FFT}.
For $\s\in \S_n$,  the mapping
$\theta_2: \s\mapsto \theta_2(\s)$ goes  as follows:
\begin{itemize}
\item Factorize $\s$ as product of disjoint 
cycles with the largest letter in the first position of each cycle.
\item Order the cycles from left to right in increasing order of their largest letters, that is  denoted by $\breve{\s}$.
\item Erase the parentheses in $\breve{\s}$ to obtain 
$\theta_2(\s)$.
\end{itemize}
For example, if  $\s=4271365$  then $\breve{\s}=(2)(4\,1)(6)(7\,5\,3)$ and  $\theta_2(\s)=2416753$.

Athanasiadis-Savvidou~\cite{AS12} 
 defined a cyclic  analogue of \emph{valley hopping}  for derangements by combining Foata's fundamental transformation and MFS-action,  see
 also Sun-Wang~\cite{SW14} and 
 Tirrell-Zhuang~\cite{TZ19}. Finally, Cooper, Jones and Zhuang~\cite{CJZ20} extended
 cyclic valley-hopping  to  permutations 
 by letting  the fixed points unchanged. 
Simply define $\psi_x(\s):\S_n\to\S_n$ by
$$
\psi_x(\s):=
\begin{cases}
\theta_2^{-1}\circ\xi_x\circ \theta_2(\s),& \text{if $x$ is not a fixed point of $\s$};\\
\s,&\text{if $x$ is a fixed point of $\s$},\\
\end{cases}
$$
where we treat the $0$-th letter of $\theta_2(\s)$ as 0 and the (n+1)-th letter as $\infty$. Given a subset $S\subseteq [n]$, define $\psi_S:=\prod_{x\in S}\psi_x$. We will call the $\Z_2^n$-action defined by the involutions $\{\psi_S\}_{S\subseteq [n]}$ \emph{cyclic valley-hopping}. For example, given $\pi=(5)(6\,4\,2)(9\,3\,7\,8)(10\,1)$ with $S=\{4,8\}$, we have $\psi_{S}(\pi)=(5)(6\,2\,4)(9\,8\,7\,3)(10\,1)$, see Fig.~\ref{graph:val-hop}.

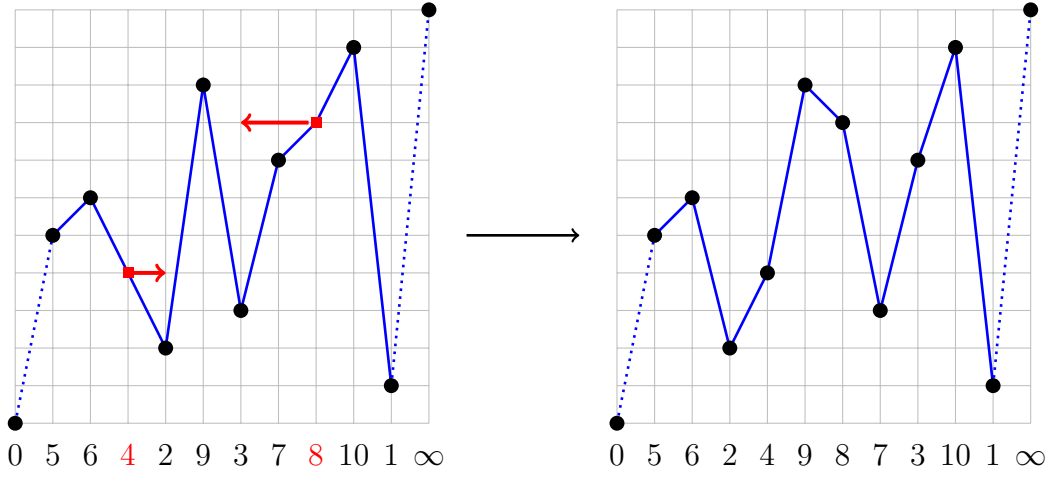
\begin{figure}[tbp]
\begin{center}
\begin{tikzpicture}[scale=0.5] 	
\draw[step=1,lightgray,thin] (0,0) grid (11,11); 
	\tikzstyle{ridge}=[draw, line width=1, color=blue] 
	\tikzstyle{ridge1}=[draw, line width=1, dotted,  color=blue] 
	\path[ridge] (1,5)--(2,6)--(3,4)--(4,2)--(5,9)--(6,3)--(7,7)--(8,8)--(9,10)--(10,1); 
	\path[ridge1] (0,0)--(1,5);
	\path[ridge1] (10,1)--(11,11);
	\tikzstyle{node0}=[circle, inner sep=2, fill=black] 
	\tikzstyle{node1}=[rectangle, inner sep=3, fill=red] 
	\tikzstyle{node2}=[rectangle, inner sep=2, fill=red] 
	\node[node0] at (0,0) {}; 
	\node[node0] at (1,5) {}; 
	\node[node0] at (2,6) {}; 
	\node[node2] at (3,4) {}; 
	\node[node0] at (4,2) {}; 
	\node[node0] at (5,9) {}; 
	\node[node0] at (6,3) {}; 
	\node[node0] at (7,7) {}; 
	\node[node2] at (8,8) {}; 
	\node[node0] at (9,10) {}; 
	\node[node0] at (10,1) {}; 
	\node[node0] at (11,11) {}; 
	\tikzstyle{hop1}=[draw, line width = 1.5, color=red,->]
	\tikzstyle{hop2}=[draw, line width = 1.5, color=red,->] 
	\path[hop1] (3,4)--(4,4);
	\path[hop2] (7.8,8)--(6,8);
	\tikzstyle{pi}=[above=-20] 
	\node[pi] at (0,0) {0}; 
	\node[pi] at (1,0) {5}; 
	\node[pi] at (2,0) {6}; 
	\node[pi,color=red] at (3,0) {4}; 
	\node[pi] at (4,0) {2}; 
	\node[pi] at (5,0) {9}; 
	\node[pi] at (6,0) {3}; 
	\node[pi] at (7,0) {7}; 
	\node[pi,color=red] at (8,0) {8}; 
	\node[pi] at (9,0) {10}; 
	\node[pi] at (10,0) {1}; 
	\node[pi] at (11,0) {$\infty$}; 
	\path[draw,line width=1,->] (12,5)--(15,5); 
	\begin{scope}[shift={(16,0)}] 
	\draw[step=1,lightgray,thin] (0,0) grid (11,11); 
	\path[ridge] (1,5)--(2,6)--(3,2)--(4,4)--(5,9)--(6,8)--(7,3)--(8,7)--(9,10)--(10,1); 
	\path[ridge1] (0,0)--(1,5);
	\path[ridge1] (10,1)--(11,11);
	\node[node0] at (0,0) {}; 
	\node[node0] at (1,5) {}; 
	\node[node0] at (2,6) {}; 
	\node[node0] at (3,2) {}; 
	\node[node0] at (4,4) {}; 
	\node[node0] at (5,9) {}; 
	\node[node0] at (6,8) {}; 
	\node[node0] at (7,3) {}; 
	\node[node0] at (8,7) {}; 
	\node[node0] at (9,10) {}; 
	\node[node0] at (10,1) {}; 
	\node[node0] at (11,11) {};
	\node[pi] at (0,0) {0}; 
	\node[pi] at (1,0) {5}; 
	\node[pi] at (2,0) {6}; 
	\node[pi] at (3,0) {2}; 
	\node[pi] at (4,0) {4}; 
	\node[pi] at (5,0) {9}; 
	\node[pi] at (6,0) {8};
	\node[pi] at (7,0) {7}; 
	\node[pi] at (8,0) {3}; 
	\node[pi] at (9,0) {10}; 
	\node[pi] at (10,0) {1};
	\node[pi] at (11,0) {$\infty$}; 
	\end{scope}
\end{tikzpicture}
\end{center}

\caption{Cyclic valley-hopping on $\pi=(5)(6\,4\,2)(9\,3\,7\,8)(10\,1)$ with $S=\{4,8\}$
yields $\psi_{S}(\pi)=(5)(6\,2\,4)(9\,8\,7\,3)(10\,1)$.}\label{graph:val-hop}
\end{figure}

Note that cyclic valley-hopping dose not affect cyclic valleys, cyclic peaks and fixed points, but toggles between cyclic double ascents and cyclic double descents. The following result is due to Cooper et al.~\cite[ Proposition 3 and 4]{CJZ20}.

\begin{prop}[Cooper, Jones and Zhuang]\label{prop:psi}
For any $S\subseteq[n]$ and $\s\in\S_n$, we have
\begin{itemize}
\item[(1)] $\Cval(\psi_S(\s))=\Cval(\s)$;
\item[(2)] $\Cpk(\psi_S(\s))=\Cpk(\s)$;
\item[(3)] $\Cda(\psi_S(\s))=(\Cda(\s)\setminus S)\cup (S\cap \Cdd(\s))$;
\item[(4)] $\Cdd(\psi_S(\s))=(\Cdd(\s)\setminus S)\cup (S\cap \Cda(\s))$;
\item[(5)] $\Fix(\psi_S(\s))=\Fix(\s)$;
\item[(6)] $\cyc(\psi_S(\s))=\cyc(\s)$.
\end{itemize}
\end{prop}
%

\subsection{Proof of Theorem~\ref{master1}}
Given a subset $S\subseteq [n]$,  consider the group $\Z_2^n$ acting on $\S_n$ via the functions $\psi_S$. 
For any permutation $\s\in\S_n$,  denote the orbit of $\s$ under cyclic valley-hopping by
$\Orb(\s):=\{\psi_S(\s)|S\subseteq[n]\}$. The cyclic MFS-action divides the set $\S_n$ into disjoint orbits. Moreover, for $\s\in\S_n$, if $x\in S$ is a cyclic double descent of $\s$, then $x$ is a cyclic double ascent of $\psi_x(\s)$. Hence, there is a unique permutation in each orbit which has no cyclic double descents. Let $\bar{\s}$ be such a unique element in $\Orb(\s)$. By Proposition~\ref{prop:psi}, we have 
\begin{subequations}
\begin{equation}\label{Orb:u}
\sum_{\pi\in\Orb(\s)}(u_1u_2)^{\cpk(\pi)}u_3^{\cda(\pi)}u_4^{\cdd(\pi)}t^{\fix(\pi)}\a^{\cyc(\pi)}=
(u_1u_2)^{\cpk(\bar{\s})}(u_3+u_4)^{\cda(\bar{\s})}t^{\fix(\bar{\s})}\a^{\cyc(\bar{\s})}.
\end{equation}
By definition, it is clear that for $\s\in\S_n$,
\begin{equation}\label{exc-drop}
   \exc(\s)=\cda(\s)+\cpk(\s)\quad \text{and} \quad \drop(\s)=\cdd(\s)+\cval(\s). 
\end{equation}
Setting $u_1=u_3=x$ and $u_2=u_4=y$ in~\eqref{Orb:u} yields
\begin{equation}\label{Orb:x}
\sum_{\pi\in\Orb(\s)}x^{\exc(\pi)}y^{\drop(\pi)}t^{\fix(\pi)}\a^{\cyc(\pi)}=(xy)^{\cpk(\bar{\s})}(x+y)^{\cda(\bar{\s})}t^{\fix(\bar{\s})}\a^{\cyc(\bar{\s})}.
\end{equation}
Thus, if $u_1u_2=xy$ and $u_3+u_4=x+y$, combining 
\eqref{Orb:u} and \eqref{Orb:x} we have 
\begin{equation}
\sum_{\pi\in\Orb(\s)}(u_1u_2)^{\cpk(\pi)}u_3^{\cda(\pi)}u_4^{\cdd(\pi)}t^{\fix(\pi)}\a^{\cyc(\pi)}=\sum_{\pi\in\Orb(\s)}x^{\exc(\pi)}y^{\drop(\pi)}t^{\fix(\pi)}\a^{\cyc(\pi)}. 
\end{equation}
\end{subequations}
 Then summing over all the orbits of $\S_n$ gives Eq.~\eqref{eq:master theorem}.\qed

\subsection{Proof of Theorem~\ref{master2}}
For a  finite set  of positive integers  $E$,  we denote by $\S_E$ 
the set of permutations of $E$.  For   $\s\in \S_E$  define the weight function
\begin{subequations}
\begin{align}
&w(\s; \a,a, b)=(u_1 u_2)^{\cpk(\s)}u_3^{\cda(\s)}u_4^{\cdd(\s)} a^{\fix(\s)}b^{\cyc(\s)-\fix(\s)}\a^{\cyc(\s)},\\
&w_1(\s; \a,a, b,t)=w(\s; \a,a, b)\,t^{\lrmaxda(\s)},\\
&w_2(\tau; \a,a, b,t)=w(\s; \a,a, b)\,t^{\rlmaxdd(\tau)}.
\end{align}
\end{subequations}
Let 
$\S_n^{\bigstar}$ be the set of permutations in $\S_n$ of which each cycle has a color in $\{\red{Red}, \blue{Blue}\}$.
An element of $\S_n^{\bigstar}$ is called a \emph{cycle decorated permutation} (CDP).
Hence,  a CDP  $\pi\in \S_n^{\bigstar}$ is in bijection 
with a pair of permutations 
 $(\s, \tau)\in \S_A\times \S_B$ such that $(A,B)$ is an ordered set partition of $[n]$ and $\pi=\s \tau$, namely,
the permutation $\s$ consists of 
all red cycles and $\tau$ the blue ones.

By Theorem~\ref{master1}, under the assumption that $xy=u_1u_2$ and  $x+y=u_3+u_4$, 
the left-hand side of equality \eqref{equ:A(u,s,a,b)-A^cyc(x,y,t,a)-with-s} can be written as 
\begin{align}\label{equ:A_n^cyc with weight function}
A_n^{\cyc}&\left( x,y, t(\a u_3+\beta u_4)(\a f+\beta g)^{-1}\,|\, \a f+\beta g\right)\nonumber\\
&=\sum_{\pi\in\S_n}(u_1 u_2)^{\cpk(\pi)}u_3^{\cda(\pi)}u_4^{\cdd(\pi)}(t(\a u_3+\be u_4))^{\fix(\pi)}(\alpha f+\be g)^{\cyc(\pi)-\fix(\pi)}\nonumber\\
&=\sum_{(\s,\tau)\in \S_n^{\bigstar}}w_1(\s; \a, u_3, f ,t)\,w_2(\tau; \beta, u_4, g,t).
\end{align}
On the other hand, any permutation $w=w_1\ldots w_{n+1}\in \S_{n+1}$ 
can be written uniquely as 
$w=u\,x\,v$ with $u=w_1\ldots w_{k-1}$, $w_k=n+1$ and $v=w_{k+1}\ldots w_{n+1}$.
 Here, by convention, $u=\emptyset$ if $k=1$ and $v=\emptyset$ if $k=n+1$.

We define a mapping  $\rho:\S_n^{\bigstar}\to  \S_{n+1}$ as follows:
\begin{align}\label{mapping: rho}
(\s,\tau)\mapsto \widetilde{\pi}:=\theta_2(\s)\,x\, \theta_1(\tau) \quad \text{with}\,\,  x=n+1,
\end{align}
where $\theta_1$ and $\theta_2$ are the FFT (cf. subsection~\ref{cyclic valley-hopping} and~\ref{sec:Eulerian identity}.)
Clearly this is a bijection and it is not difficult to verify the following properties:
\begin{subequations}
\label{sta:fix-A-B}
\begin{align}
&\cpk(\s)+\cpk(\tau)=\pk(\widetilde{\pi})-1,\\
&\cda(\s)+\fix(\s)+\cda(\tau)=\lda(\widetilde{\pi}),\\
&\cdd(\s)+\cdd(\tau)+\fix(\tau)=\rdd(\widetilde{\pi}),\\
&\fix(\s)+\fix(\tau)=\lrmaxda(\widetilde{\pi})+\rlmaxdd(\widetilde{\pi}),\\
&\cyc(\s)-\fix(\s)=\fmax(\widetilde{\pi})-1,\\
&\cyc(\tau)-\fix(\tau)=\bmax(\widetilde{\pi})-1,\\
&\cyc(\s)=\LRmax(\widetilde{\pi})-1,\\
&\cyc(\tau)=\RLmax(\widetilde{\pi})-1.
\end{align}
\end{subequations}
Note that $\pk(\widetilde{\pi})-1=\val(\widetilde{\pi})$. Combining~\eqref{sta:fix-A-B} with~\eqref{equ:A_n^cyc with weight function} we obtain
 \begin{multline}
\sum_{(\s,\tau)\in \S_n^{\bigstar}}w_1(\s; \a, u_3, f,t )\,w_2(\tau; \beta, u_4, g,t)\\
=\sum_{\widetilde{\pi} \in \mathfrak{S}_{n+1}}(u_1u_2)^{{\pk}(\widetilde{\pi})}u_3^{{\lda}(\widetilde{\pi})}u_4^{{\rdd}(\widetilde{\pi})}f^{\fmax(\widetilde{\pi})-1}
g^{\bmax(\widetilde{\pi})-1}\\
\times t^{\lrmaxda(\widetilde{\pi})+\rlmaxdd(\widetilde{\pi})}{\alpha}^{{\LRmax}(\widetilde{\pi})-1}{\beta}^{{\RLmax}(\widetilde{\pi})-1},
 \end{multline}
 which  proves~\eqref{equ:A(u,s,a,b)-A^cyc(x,y,t,a)-with-s}. 
 \qed
 
 \begin{figure}[t]
\begin{tikzpicture}[scale=1.6]
\tikzstyle{node0}=[circle, inner sep=2, fill=black] 
	\tikzstyle{node1}=[rectangle, inner sep=3, fill=red] 
	\tikzstyle{node2}=[rectangle, inner sep=2, fill=red] 
\draw (0.6,0.15) rectangle (4.9,2.5);
\draw[line width=0.8,dashed] (2.7,2.5)--(2.7,0.15);

\draw[->, color=red] (1.5,1.5) to [bend left] (1.75,1.04);
\draw[->,color=red] (1.25,0.64) to [bend left] (1.45,1.5);
\draw[->,color=red] (1.77,1.02) to [bend left] (1.30,0.6);
\fill (1.5,1.5) circle (1.1pt);
\fill (1.25,0.6) circle (1.1pt);
\fill (1.75,0.99) circle (1.1pt);
\node[above] at (1.5,1.5) {$7$};
\node[below] at (1.25,0.6) {$5$};
\node[right] at (1.75,0.99) {$1$};

\fill (2.3,1.7) circle (1.1pt);
\draw[-,color=red] (2.339,1.7) to [out=-30,in=30] (2.3,1.2);
\draw [->,color=red](2.3,1.2) to [out=150,in=210] (2.26,1.69);
\node[above] at (2.3,1.7) {$9$};

\fill (1,2) circle (1.1pt);
\draw [-,color=red] (1.03,1.98) to [out=-30,in=30,] (1,1.5);
\draw [->,color=red](1,1.5) to [out=150,in=210] (0.97,1.98);
\node[above] at (1,2) {$2$};

\fill (3.3,1) circle (1.1pt);
\node[above] at (3.3,1){$8$};
\draw [-,color=blue] (3.33,0.98) to [out=-30,in=30,] (3.3,0.6);
\draw [->,color=blue](3.3,0.6) to [out=150,in=210] (3.27,0.98);

\draw[->, color=blue] (4,1.8) to [bend left] (4.5,1.24);
\draw[->,color=blue] (4.5,1.2) to [bend left] (3.85,0.94);
\draw[->,color=blue] (3.8,0.95) to [bend left] (3.97,1.76);
\fill (4,1.8) circle (1.1pt);
\fill (4.5,1.2) circle (1.1pt);
\fill (3.8,0.95) circle (1.1pt);
\node[above] at (4,1.8) {$4$};
\node[right] at (4.5,1.2) {$3$};
\node[below] at (3.8,0.95) {$6$};

\node[above] at (2.4,2.1){$\red {\sigma}$};
\node[above] at (3,2.1){$\blue{\tau}$};
\draw[-stealth] (5,1.2) to (6,1.2);

\begin{scope}[shift={(6.1,0)}]
\draw[step=0.25,lightgray,thin] (0,0) grid (2.75,2.75); 
\tikzstyle{ridge}=[draw, line width=1,  color=blue] 
\tikzstyle{ridge1}=[draw, line width=1, dotted,  color=blue] 
\path[ridge] (0.25,0.5)--(0.5,1.75)--(0.75,0.25)--(1,1.25)--(1.25,2.25)--(1.5,2.5)--(1.75,2)--(2,1)--(2.25,0.75)--(2.5,1.5); 
\path[ridge1] (0,0)--(0.25,0.5);
\path[ridge1] (2.5,1.5)--(2.75,0);
\node[node0] at (0,0) {}; 
	\node[node0] at (0.25,0.5) {}; 
	\node[node0] at (0.5,1.75) {}; 
	\node[node0] at (0.75,0.25) {}; 
	\node[node0] at (1,1.25) {}; 
	\node[node0] at (1.25,2.25) {}; 
	\node[node0] at (1.5,2.5) {}; 
	\node[node0] at (1.75,2) {}; 
	\node[node0] at (2,1) {}; 
	\node[node0] at (2.25,0.75) {}; 
	\node[node0] at (2.5,1.5) {}; 
	\node[node0] at (2.75,0) {};
	
	\node[left] at (0.25,0.5) {\red{$2$}};
	\node[above] at (0.5,1.75) {\red{$7$}}; 
	\node[right] at (0.75,0.25) {$1$}; 
	\node[right] at (1,1.25) {$5$}; 
	\node[left] at (1.25,2.25) {\red{$9$}}; 
	\node[above] at (1.5,2.5) {{\bf$10$}}; 
	\node[left] at (1.75,2) {\blue{$8$}}; 
	\node[left] at (2,1) {$4$}; 
	\node[below] at (2.25,0.75) {$3$}; 
	\node[above] at (2.5,1.5) {\blue{$6$}}; 
\end{scope}
\end{tikzpicture}
\caption{The mapping  $\rho: (\sigma,\tau)\mapsto \widetilde{\pi}=\red{2\,7\,1\,5\,9}\,{\bf10}\,\blue{8\,4\,3\,6}.$ }

\label{configuration graph}
\end{figure}
  \begin{example}
Consider the pair $(\s,\tau)\in\S_9^{\bigstar}$, where
$\s=\red{(2)(9)(5\,7\,1)}$ and $\tau=\blue{(6\,4\,3)(8)}$,
then $\theta_2(\s)=27159$ and   $\theta_1(\tau)=8436$.  Thus $\widetilde{\pi}=2\,7\,1\,5\,9\,\mathbf{10}\,8\,4\,3\,6 \in\S_{10}$, see Fig.~\ref{configuration graph}. 

As a check of \eqref{sta:fix-A-B},  we have:
\begin{itemize}
    \item at the left-hand side,
$\cpk(\s)=|\{7\}|=1,\;\cpk(\tau)=|\{6\}|=1,\; \cda(\s)=|\{5\}|=1$, $\cdd(\tau)=|\{4\}|=1$, $\fix(\s)=|\{2,9\}|=2,\; \fix(\tau)=|\{8\}|=1,\,\cyc(\s)=3$ and $\cyc(\tau)=2$; 
\item at the right-hand side, $\pk(\widetilde{\pi})=|\{6,7,10\}|=3$, $\val(\widetilde{\pi})=|\{1, 3\}|=2$, $ \lda(\widetilde{\pi})=|\{2,5,9\}|=3,\; \rdd(\widetilde{\pi})=|\{4,8\}|=2$, $
\LRmax(\widetilde{\pi})=|\{2,7,9,10\}|=4, \RLmax(\widetilde{\pi})=|\{6,8,10\}|=3$, 
$\lrmaxda(\widetilde{\pi})=|\{2,9\}|=2$,
$\rlmaxdd(\widetilde{\pi})=|\{8\}|=1$,
$\fmax(\widetilde{\pi})=|\{7,10\}|=2$ and $\bmax(\widetilde{\pi})=|\{6,10\}|=2$. 
\end{itemize}
\end{example} 
\section{Cycle André and André permutations: Proof of Theorem~\ref{main-theorem1}}
\subsection{Counting André permutations of type 1 by des and lmax}
When $x=t=1$, let
 $\gamma_{n,j}(\a):=\gamma_{n,j} (\a, 1)$ and 
\begin{subequations}
\begin{equation}
A_{n,k}(\a):=\sum_{
\mycom{\s\in \S_{n+1}}{\des(\s)=k}}\alpha^{\lrmax(\s)+\rlmax(\s)-2},
\end{equation}
then, Eq.~\eqref{gamma:A} can be rephrased as
\begin{equation}\label{gamma-alpha}
    \sum_{k=0}^nA_{n,k}(\a)y^k=\sum_{j=0}^{\lfloor n/2\rfloor}\gamma_{n,j}(\a)y^j(1+y)^{n-2j}.
\end{equation}
Considering the position of letter 1 of permutations in $\S_{n+2}$,  it is easy to verify that $A_{n,k}(\a)$ satisfy the following recurrence,
\begin{equation}\label{rec:A_{n,k}(a)}
A_{n+1,k}(\a)=(k+\a)A_{n,k}(\a)+(n+1-k+\a)A_{n,k-1}(\a),
\end{equation}
\end{subequations}
where $A_{0,0}(\a)=1$ and $A_{0,-1}(\a)=0$, for $0\le k\le n$.

\begin{lemma}
For $n\ge0$ and $0\le j\le\lfloor n/2\rfloor$,  the $\gamma$-coefficients 
$\gamma_{n,j}(\a)$ satisfy the  recurrence relation 
\begin{equation}\label{rec:gamma}
\g_{n+1,j}(\a)=(\a+j)\g_{n,j}(\a)+2(n+2-2j)\g_{n,j-1}(\a),
\end{equation}
with  $\g_{0,0}(\a)=1$ and $\g_{0,-1}(\a)=0$. Equivalently, setting $d_{n+1,j}(\a)=\g_{n+1,j}(\a)/2^j$, then
\begin{equation}\label{recurrence:d_n}
d_{n+1,j}(\a)=(\a+j)d_{n,j}(\a)+(n+2-2j)d_{n,j-1}(\a),
\end{equation}
with $d_{0,0}(\a)=1$ and $d_{0,-1}(\a)=0$.
\end{lemma}
\begin{proof}
We verify recurrence~\eqref{rec:gamma} by induction on $n$. Since  $\g_{2,0}(\a)=\a^2$, $\g_{1,0}(\a)=\a$ and $\g_{1,-1}(\a)=0$, this  is obvious when $n=1$.  
If $n\ge 2$ and $0\le k\le n$, 
Eq.~\eqref{gamma-alpha} is equivalent to
\begin{equation}\label{relation:A_n-gamma}
A_{n,k}(\a)=\sum_{0\le j\le k}\binom{n-2j}{k-j}\g_{n,j}(\a).
\end{equation}
Substituting~\eqref{relation:A_n-gamma} in~\eqref{rec:A_{n,k}(a)} and using~\eqref{rec:gamma}, we derive 
\begin{align}
\sum_{0\le j\le k}&\binom{n+1-2j}{k-j}\bigg((\a+j)\g_{n,j}(\a)+2(n+2-2j)\g_{n,j-1}(\a)\bigg)\nonumber\\
&=\bigg((k+\a)\sum_{0\le j\le k}\binom{n-2j}{k-j}+(n+1-k+\a)\sum_{0\le j\le k}\binom{n-2j}{k-1-j}\bigg)\g_{n,j}(\a).
\end{align}
Extracting the coefficients of $\g_{n,j}(\a)$ we obtain:
\begin{align}
(\a+j)\binom{n+1-2j}{k-j}&+2(n-2j)\binom{n-1-2j}{k-j-1}\nonumber\\
&=(k+\a)\binom{n-2j}{k-j}+(n+1-k+\a)\binom{n-2j}{k-1-j},
\end{align}
which is obvious by Maple.
\end{proof}
\begin{lemma}\label{lem:1-des-rlminda}
For $n\ge 0$,  we have
\begin{equation}
\sum_{j=0}^{\lfloor n/2\rfloor}d_{n,j}(\a)x^j=\sum_{\s\in\A_{n+1}^1}x^{\des(\s)}\a^{\lrmax(\s)-1}.
\end{equation}
\end{lemma}

\begin{proof}
Let 
\begin{equation}
\sum_{\s\in\A_{n+1}^1}x^{\des(\s)}\a^{\lrmax(\s)-1}=\sum_{j=0}^{\lfloor n/2\rfloor}d_{n,j}^{\ast}(\a)x^j.
\end{equation}
We shall prove that $d_{n,j}^{\ast}(\a)$ satisfies the same recurrence relation~\eqref{recurrence:d_n}, i.e., 
\begin{equation}\label{recurrence:d_n^ast}
d_{n+1,j}^{\ast}(\a)=(\a+j)d_{n,j}^{\ast}(\a)+(n+2-2j)d_{n,j-1}^{\ast}(\a).
\end{equation}

The left-hand side of equation~\eqref{recurrence:d_n^ast} counts permutations in $\s\in\A_{n+2}^1$ with $j$ descents considering the weights of $\a^{\lrmax-1}$. By the fact that deleting the letter 1 in $\s\in\A_{n+2}^1$ yields also an André permutation of $\{2,3,\dots,n+2\}$. Thus, we can obtain such a permutation uniquely in one of two ways. For the first way, choose a permutation $\s=\s_1\cdots\s_{n+1}\in\A_{E}^1$ with $j$ descents where $E=\{2,3,\dots,n+2\}$, and insert $1$ after $\s_i$ if $i$ is descent, or insert 1 at the beginning. Note that letter 1 will contribute a left-to-right maximum at the beginning, while in other positions not. Thus, there are $j+1$ ways to insert 1, so we obtain by this factor $(\a+j)d_{n,j}^{\ast}(\a)$. For the second way, choose $\s=\s_1\cdots\s_{n+1}\in\A_{E}^1$ with $j-1$ descents where $E=\{2,3,\dots,n+2\}$, and insert 1 after $\s_i$ if $i$ is an ascent and $i-1$ is not a descent. Since an André permutation has no double descents, thus there are $n-2(j-1)$ ways to insert 1, so we obtain the second factor $(n+2-2j)d_{n,j-1}^{\ast}(\a)$. 
\end{proof}

Combining Eq.~\eqref{rec:gamma}, ~\eqref{recurrence:d_n} and Lemma~\ref{lem:1-des-rlminda},  we derive the following exponential generating function from~\eqref{equ2:gen-Gamma-coe}. 
\begin{theorem}
We have
\begin{equation}\label{exp-gen-func-F(x,a)}
\sum_{n\ge 0}\sum_{\s\in\A_{n+1}^1}x^{\des(\s)}\a^{\lrmax(\s)-1}\frac{z^n}{n!}=\bigg(\frac{\sqrt{2x-1}\sec(\frac{z}{2}\sqrt{2x-1})}{\sqrt{2x-1}-\tan(\frac{z}{2}\sqrt{2x-1})}\bigg)^{2\a}.
\end{equation}
\end{theorem}

\begin{remark}
When $\a=1$, \eqref{exp-gen-func-F(x,a)} was first proved by Foata and Schützenberger~\cite{FSch73},  other proofs are given  by Foata and Han~\cite{FH01}, and 
Chow and Shiu~\cite[Theorem~2.1]{CS11}.
\end{remark}
\begin{remark}
It would be interesting to extend the above approach by 
incorporating a suitable  parameter $t$. 
\end{remark}

\subsection{Counting  André cycles by drop}


The following is an alternative definition of André permutations of the first kind.
\begin{lemma}[\cite{FH16}]\label{lem:def-andre}
{\rm 
Say that the empty word e and each one-letter word are both André permutations of the first kind. Next, if $w=w_1w_2\cdots w_n$ $(n\ge 2)$ is a permutation of a set of positive integers $Y=\{y_1<y_2\cdots<y_n\}$, write $w=v\min(w)v'$. Then, $w$ is an \textbf{André permutation of the first kind}, if both $v$ and $v'$ are André permutation of the first kind, and furthermore if $\max(vv')$ is a letter of $v'$.
}
\end{lemma}
Define 
 the exponential generating function
 \begin{subequations}
\begin{align}\label{exp:E(x;z)}
E(x; z):=\sum_{n\ge 0}\left(\sum_{\s\in\A_{n+1}^1}x^{\des(\s)}\right)\frac{z^n}{n!}.
\end{align}
Then,  Eq.~\eqref{exp-gen-func-F(x,a)} with $\a=1$ yields
\begin{equation}\label{exp-gen-E(x,z)2}
E(x;z)=\bigg(\frac{\sqrt{2x-1}\sec(\frac{z}{2}\sqrt{2x-1})}{\sqrt{2x-1}-\tan(\frac{z}{2}\sqrt{2x-1})}\bigg)^{2}.
\end{equation}
We give a combinatorial proof of the following lemma, which  can also be derived from~\eqref{exp-gen-E(x,z)2}.
\begin{lemma}\label{lem:first-differentiating} We have 
\begin{equation}\label{equ:E(x;z)}
E(x;z)=\exp\bigg(z+x\int_0^z\bigg(\int_0^vE(x;u)\,du\bigg)dv\bigg).
\end{equation}
\end{lemma}
\end{subequations}
\begin{proof}
Let 
\begin{subequations}
\begin{equation}
E_n(x):=\sum_{\s\in\A_{n}^1}x^{\des(\s)}.
\end{equation}
By Lemma~\ref{lem:def-andre}, 
any permutation $w\in \A_{n+2}^1$ can be written as 
$w=v\textbf{1} v'$ where $v$ and $v'$ are André permutations of the first kind. 
Let  $j+1$ be the position of $1$. Then,  classifying the permutations in $\A_{n+2}^1$ by  the index $j\in\{0,1,\dots,n\}$, we have
\begin{equation}\label{equ:recurrence-x}
E_{n+2}(x)=E_{n+1}(x)+x\sum_{j=1}^{n}\binom{n}{j}E_{j}(x)E_{n+1-j}(x),
\end{equation}
in other words,
\begin{equation}\label{equ:recurrence relation-E}
\sum_{n\ge 0}E_{n+2}(x)\frac{z^n}{n!}=\bigg(\sum_{n\ge 0}E_{n+1}(x)\frac{z^n}{n!}\bigg)\bigg(1+x\sum_{n\ge 1}E_{n}(x)\frac{z^n}{n!}\bigg).
\end{equation}
As $E(x;z)=\sum_{n\ge 0}E_{n+1}(x)\frac{z^n}{n!}$, the above equation can be written as 
\begin{equation}\label{partial-E(x;z)}
\frac{\partial E(x; z)} {\partial z} =E(x; z) \left(1+x\int_0^zE(x; u)du\right), 
\end{equation}
\end{subequations}
which is the equation obtained by 
 differentiating the two sides of \eqref{equ:E(x;z)} with respect to $z$. 
 As the two sides of \eqref{equ:E(x;z)} agree at $z=0$,  we are done.
\end{proof}

 Recall that 
a cycle $C=(a_1,\dots, a_k)$ of an ordered set $A$ is a  
an
\textbf{\emph{André cycle}} if  $a_1=\min\{a_1,\dots, a_k\}$ and 
 $a_2\ldots a_k$ is an André permutation of the first kind. 
Let $\AC_n$ be the set of  André  cycles of $[n]$.

\begin{lemma}\label{lem:cyclic-andre-drop}
We have
\begin{equation}\label{equ:exp-CA-drop}
\sum_{n\ge 2}\sum_{\s\in\AC_{n}}x^{\drop(\s)}\frac{z^n}{n!}=x\int_0^z\bigg(\int_0^tE(x;u)\,du\bigg)dt.
\end{equation}
\end{lemma}
\begin{proof}
For $\s=(1,a_2,\dots,a_{n-1}, n)\in\AC_n$ with $n\ge 2$,
let $\zeta(\s)=(a_2-1)(a_3-1)\cdots (n-1)\in\A_{n-1}^1$. It is clear that 
the mapping $\zeta: \AC_n\to\A_{n-1}^1$  a bijection satisfying
\begin{equation}
\drop(\s)=\des(\zeta(\s))+1.
\end{equation}
Therefore, 
\begin{align}
\sum_{n\ge 2}\sum_{\s\in\AC_{n}}x^{\drop(\s)}\frac{z^n}{n!}&=\sum_{n\ge 2}\sum_{\s\in\A_{n-1}^1}x^{\des(\s)+1}\frac{z^n}{n!}\nonumber\\
&=x\sum_{n\ge 0}E_{n+1}(x)\frac{z^{n+2}}{(n+2)!}.
\end{align}
By \eqref{exp:E(x;z)}, the last  series is equal to the right-hand side of \eqref{equ:exp-CA-drop}.
\end{proof}
\subsection{Proof of Theorem~\ref{main-theorem1}~\eqref{combinatorial interpretation of d-web}}
Recall  that a  permutation is a cycle André permutation
if it is a product of disjoint André cycles and the set of cycle André permutations of $[n]$ is denoted by $\Web_n$.
By Lemma~\ref{lem:cyclic-andre-drop} we have 
\begin{align}
   \sum_{n\ge 1}\sum_{\s\in \AC_n}x^{\drop(\s)}t^{\fix(\s)}
   \frac{z^n}{n!}&=
tz+\sum_{n\ge2}\sum_{\s\in\AC_n}x^{\drop(\s)}\frac{z^n}{n!}\nonumber\\
&=(t-1)z+\left(z+x\int_0^z\bigg(\int_0^tE(x;u)\,du\bigg)dt\right).
\end{align}
Combining the exponential formula and Lemma~\ref{lem:first-differentiating} we obtain
 \begin{align}\label{exponential-formula-web}
\sum_{n\ge 0}\sum_{\s\in\Web_n}x^{\drop(\s)}t^{\fix(\s)}\a^{\cyc(\s)}\frac{z^n}{n!}
=e^{\a (t-1)z} \left(E(x;z)\right)^\a.
\end{align}
Substituting \eqref{exp-gen-E(x,z)2} in \eqref{exponential-formula-web} yields
\begin{align}\label{egf-formula-web}
\sum_{n\ge 0}\sum_{\s\in\Web_n}x^{\drop(\s)}t^{\fix(\s)}\a^{\cyc(\s)}\frac{z^n}{n!}=e^{(t-1)\a z}\bigg(\frac{\sqrt{2x-1}\sec(\frac{z}{2}\sqrt{2x-1})}{\sqrt{2x-1}-\tan(\frac{z}{2}\sqrt{2x-1})}\bigg)^{2\a}.
\end{align}
On the other hand, from  Theorems~\ref{coro: gamma-expansion-A_n(x,y|a)-cyc-version} and ~\ref{Thm: egf of gamma coefficients}
  we derive the  generating function

      \begin{equation}\label{exp-dn}
\sum_{n\ge 0}\sum_{j=0}^{\lfloor \frac{n}{2}\rfloor}d_{n,j}(\a,t)x^j\frac{z^n}{n!}=
e^{(t-1)\a z}\bigg(\frac{\sqrt{2x-1}\sec(\frac{z}{2}\sqrt{2x-1})}{\sqrt{2x-1}-\tan(\frac{z}{2}\sqrt{2x-1})}\bigg)^{2\a}.
\end{equation}
Eq.~\eqref{combinatorial interpretation of d-web} then follows by comparing the above two generating functions.
\qed
\subsection{Proof of Theorem~\ref{main-theorem1}~\eqref{combinatorial interpretation of d-andre}}\label{mapping}

We  need a mapping $\phi: \Web_n\to \A_{n+1}^1$ defined as follows: let $\s=C_1C_2\ldots C_l\in \Web_n$, where 
  $C_1, \ldots, C_l$ are André cycles, ordered from   left to right in increasing order of their smallest  letters, which are  at the end of each cycle;
erasing the parentheses   in $\s$ and appending letter $n+1$ to the end results in $\phi(\s)$. 

It is not difficult  to verify that $\phi$ is a bijection (See, for example,  the proof of Theorem~4.1 in Hwang-Jang-Oh~\cite{HJO23}). Moreover, 
similar to Lemma~\ref{mapping1-properties}, we have 
\begin{equation}\label{prop:phi} 
(\drop,\fix,\cyc)\,\s=
(\des,\rlminda,\rlmin-1)\,\phi(\s).
\end{equation}
Thus, combining~\eqref{prop:phi}  with \eqref{combinatorial interpretation of d-web} we prove \eqref{combinatorial interpretation of d-andre}.\qed

\begin{example}
 If  $\s=(5,6,1)(7,4,8,2)(3)\in\Web_{8}$, then $\s':=\phi(\s)=561748239\in\A_{9}^1$. 
Moreover,
 \begin{itemize}
     \item $\drop(\s)=|\{6,7,8\}|=3$, $\fix(\s)=|\{3\}|=1$ and  $\cyc(\s)=3$;
 \item  $\des(\s')=|\{2,3,6\}|=3$, $\rlminda(\s')=|\{3\}|=1$ and $\rlmin(\s')=|\{1,2,3,9\}|=4$.
 \end{itemize}
\end{example}

\subsection{Proof of Theorem~\ref{main-theorem1}\eqref{combinatorial interpretation of d-andre2}}
We shall construct 
a bijection $\Phi: \A_n^1\to\A_n^2$,  which preserves the triple statistic $(\des, \rlminda, \rlmin)$. 
\subsubsection{Binary and  André trees}

A binary tree $T$ on a finite ordered set $S$ is  (recursively) defined as follows:
\begin{itemize}
    \item if $S=\emptyset$, then $T=\emptyset$,
    \item otherwise, $T$ is a triple 
$(T_1, r, T_2)$ with $r\in S$ and $T_1$ and $T_2$ are binary trees on $S_1$ and $S_2$, respectively such that $(\{r\}, S_1, S_2)$ is a partition of $S$ (possibly with empty $S_1$ or $S_2$). 
\end{itemize}
For the binary tree $T=(T_1, r, T_2)$ on $S$, the element $r$ is called \emph{root} of the tree, $S$ the set of \emph{nodes} or \emph{vertices} of $T$, $T_1$ (resp. $T_2$) is the left  (resp. right) subtree of $T$. A binary tree is called an increasing binary tree, if the vertices along any path from the root are increasing. A vertex of $T$ without successor is called \emph{leaf}, otherwise it is an \emph{interior vertex}. Let $\leaf(T)$ be the number of leaves of $T$. 

\begin{definition}\cite{Dis14}\label{def:andre:trees} 
    An increasing binary tree $T$ on a finite ordered set is an André tree of the first kind (resp. second kind), if 
    \begin{itemize}
    \item no vertex of $T$ has only left successor;
\item and if $x$ is an interior vertex of $T$ with left successor $y$ and right successor $z$, then $\max(T(y))<\max(T(z))$ (resp. $y>z$), where $T(a)$ denotes the the subtree rooted at $a$ and $\max(T(a))$ is the largest vertex of $T(a)$.
\end{itemize}
The set of André trees of the first kind (resp. second kind) on $[n]$ is denoted $\T_n^1$ (resp. $\T_n^2$).
\end{definition}
\begin{figure}[tp]
\centering
\includegraphics[scale=0.25]{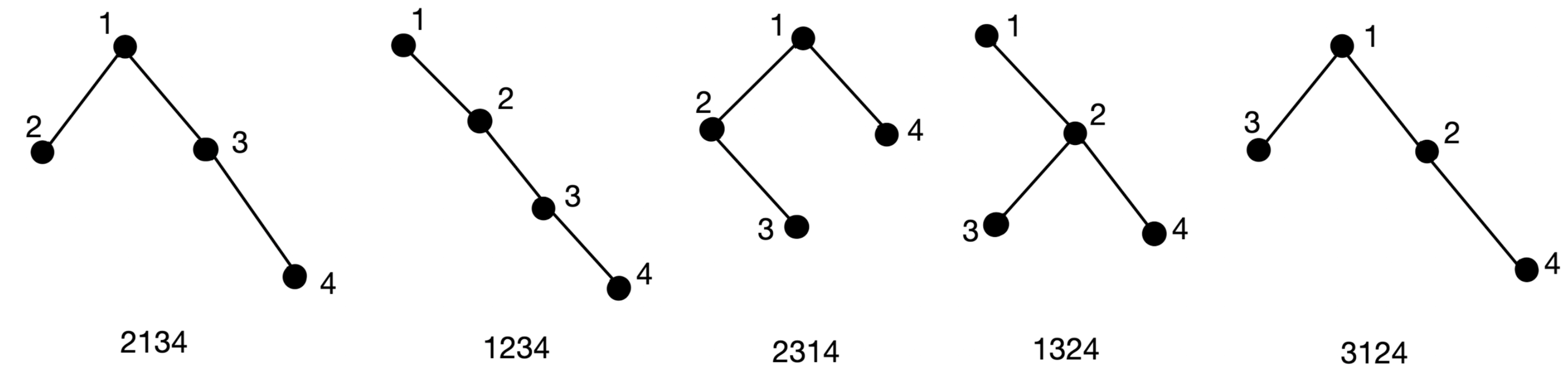}
\caption{The André trees in $\T_4^1$ and associated André permutations in $\A_4^1$.}\label{example andré1}
\end{figure}
\begin{figure}[tp]
\centering
\includegraphics[scale=0.25]{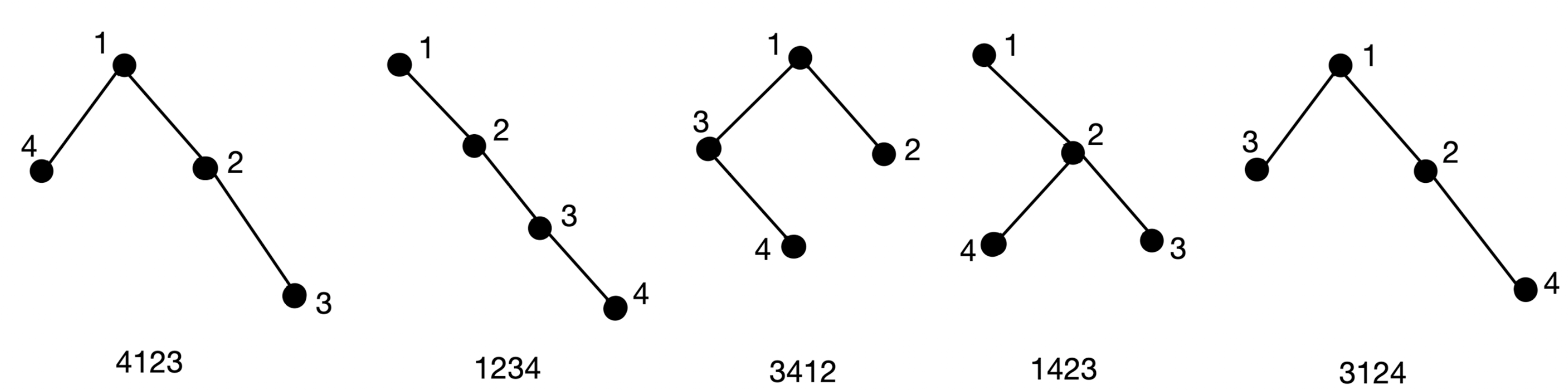}
\caption{The André trees in $\T_4^2$ and associated André permutations in $\A_4^2$.}\label{example andré2}
\end{figure}

 There is  a well-known mapping
sending permutations to increasing binary trees~\cite[Section~1.5]{St12} that we recall as in the following.  Let $\s=\s_1\dots\s_n$ be a permutation on $[n]$. Define inductively a  binary tree $T_{\s}$ as follows. 
 Let $x=\s_i$ be the least letter of $\s$ and factor  $\s$ as $\s=u\,x\,v$, where $u=\s_1\ldots \s_{i-1}$ and $v=\s_{i+1}\ldots \s_n$. 
 Now, let $T_{\s}=(T_u,x,T_v)$, where 
  $x$ is the root of $T_{\s}$, and $T_{u}$ and $T_{v}$ are  the left and right subtrees of $x$.

The mapping  $\omega: \s\mapsto T_{\s}$ is a bijection  from  $\S_n$ 
to increasing binary trees on $[n]$.   
To recover $\s$ from $T_{\s}$ we use the 
symmetric order traversal (inorder traversal), i.e.,  a tour of the vertices  of $T_{\s}$ obtained by applying the recursive algorithm: visit in symmetric order the left subtree of the root (if it exists); visit the root of the tree; visit in symmetric order the right subtree of the root (if it exists). 
\begin{example}
The André trees in  $\T_4^1$ or $\T_4^2$ and their associated  permutations are given in Fig.~\ref{example andré1} and~\ref{example andré2}.   
\end{example}

Given an increasing binary tree $T$, the \emph{right face} of  $T$ is the path from the root   to a leaf by performing only right-steps. Denote by $\rsho(T)$ the number of vertices in the right face of $T$. Let $\rsho'(T)$ be the number of interior vertices with only right successor in the right face of $T$. For example, in Fig.~\ref{example-rshou}, we have $\rsho'(T_{\s})=|\{4,8\}|=2$ and $\rsho(T_{\s})=|\{1,3,4,5,8,14\}|=6$.

It is easy to verify the following properties, some of them are in~\cite[Setion~1.5]{St12}.

\begin{lemma}\label{lem:binary-trees}
    Let $\s=\s_1\dots\s_n\in\S_n$ and $T_{\s}:=\omega(\s)$, with $\s_0=\s_{n+1}=0$. Then,
\begin{itemize}
    \item  letter $i$ is a double ascent of $\s$, iff vertex  $i$ of $T_{\s}$ has only  right successor;
    \item  letter $i$ is a double descent of $\s$, iff vertex $i$ of $T_{\s}$ has only  left successor;
    \item  letter $i$ is a valley of $\s$, iff vertex $i$ of $T_{\s}$ has  left and right successors;
    \item  letter $i$ is a peak of $\s$, iff vertex $i$ of $T_{\s}$ is a leaf;
    \item   letter $i$ is a right-to-left minimum of $\s$, iff vertex $i$ is in the right face of $T_{\s}$.
    \end{itemize}
\end{lemma}




\begin{proposition}\label{prop:statistics-keep}
For $i=1,2$, the mapping $\omega:\A_{n}^i\to\T_n^i$ is a bijection. Moreover, 
\begin{equation}
(\des,\rlminda,\rlmin)\,\s =(\leaf-1, \rsho',\rsho)\,T_{\s},
\label{andre1-triple}
\end{equation}
where  $\s\in\A_{n}^i$ and  
 $T_{\s}=\omega(\s)$.
\end{proposition}
\begin{proof}  
In Definition~\ref{def:x-factorization-André1-2}, we define  André permutations by the so-called $x$-factorizations of permutations.
In view of Definition~\ref{def:andre:trees} and Lemma~\ref{lem:binary-trees}, it is clear (see~\cite{Dis14,HR98,FH01})
that 
restricting the mapping $\omega$ on $\A_n^1$ (resp. $\A_n^2$) sets a bijection from $\A_n^1$ (resp. $\A_n^2$) to $\T_n^1$ (resp. $\T_n^2$). It remains to verify 
\eqref{andre1-triple}. We only prove the $i=1$ case. 
If $\s\in\A_n^1$, 
then 
$\dd(\s)=0$, so (see \eqref{relation:V and W})
\[
\des(\s)=\val(\s)  +\rdd(\s)=
\pk(\s)-1=\leaf(T_{\s})-1;
\]  
besides, it follows from  Lemma~\ref{lem:binary-trees} that   $\rlmin(\s)=\rsho(T_{\s})$ and $\rlminda(\s)=\rsho'(T_{\s})$.
\end{proof}

\begin{figure}[tp]
\centering
\includegraphics[scale=0.12]{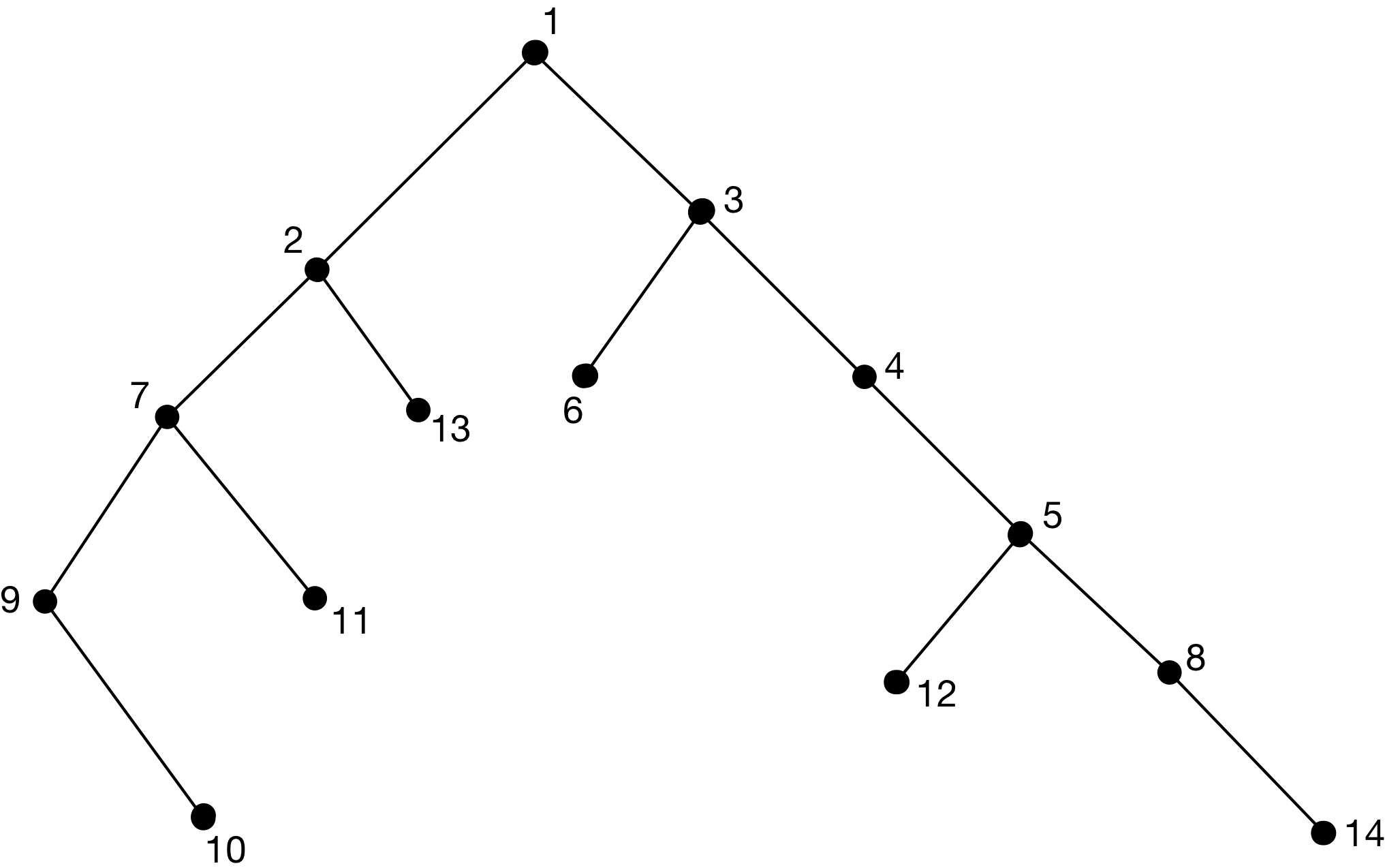}
\caption{The binary tree $T_\s$ with $\s={9\,10\,7\,11\,2\,13\,1\,6\,3\,4\,12\,5\,8\,14}\in\A_{14}^1$.}\label{example-rshou}
\end{figure}
\begin{example}
    If $\s=9\,10\,7\,11\,2\,13\,1\,6\,3\,4\,12\,5\,8\,14\in\A_{14}^1$, then the binary tree $T_{\s}$ is given in Fig.~\ref{example-rshou}. 
    Moreover,
    \begin{itemize}
        \item $\des(\s)=|\{2,4,6,8,11\}|=5$, $\rlminda(\s)=|\{4,8\}|=2$,  and\\
        $\rlmin(\s)=|\{1,3,4,5,8,14\}|=6$;
        \item  $\leaf(T_{\s})=|\{6,10,11,12,13,14\}|=6$, $\rsho'(T_{\s})=|\{4,8\}|=2$, and \\
        $\rsho(T_{\s})=|\{1,3,4,5,8,14\}|=6$.
         \end{itemize}
\end{example}

\subsubsection{Construction of bijection $\Phi$} First of all,
for $\s=\s_1\dots\s_n\in \A_n^1$, let $T_\s=\omega(\s)\in \T_n^1$ and
 \begin{equation}\label{S-T-set}   
S_{T}=\{i:  \text{vertex $\s_i$ has  2 successors in  $T_{\s}$} \}.
  \end{equation}

 If $i\in S_T$, let 
  $\s_j$ and $\s_k$ be the left  and  right successors of $\s_i$.   We define operator $\phi_i$ on the vertices of $T_\s$ as in the following:
\begin{itemize}
    \item if $\s_j>\s_k$,  then  $\phi_iT_{\s}=T_{\s}$; 
    \item if $\s_j<\s_k$. then exchange $\s_j$ with $\max(T(\s_k))$, and rearrange the vertices of $T(\s_j)$ and $T(\s_k)$ so that they keep the same relative order, respectively. 
\end{itemize}
\begin{example}
If $\s=4\,7\,2\,8\,1\,3\,6\,5\,9\,10$, then  $T_{\s}=\omega(\s)$ is shown at the left of  
    Fig.~\ref{example: phi}. 
    As $\s_5=1$ we have $\phi_5(T_{\s})=\omega(\tau)$ with 
$\tau=7\,8\,4\,10\,1\,2\,5\,3\,6\,9$.
\end{example}

\begin{figure}[tp]
\centering
\includegraphics[scale=0.1]{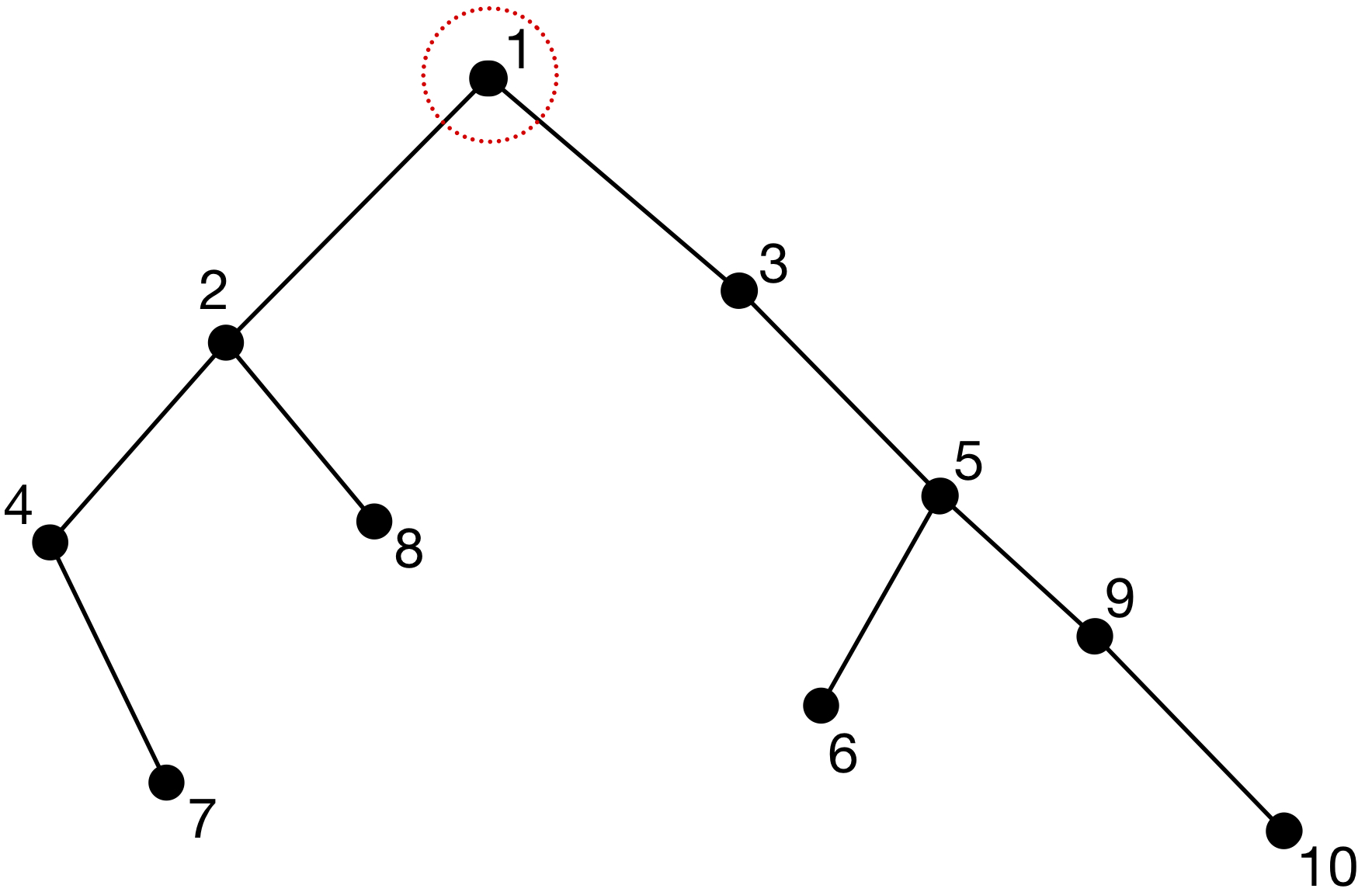}
\includegraphics[scale=0.1]{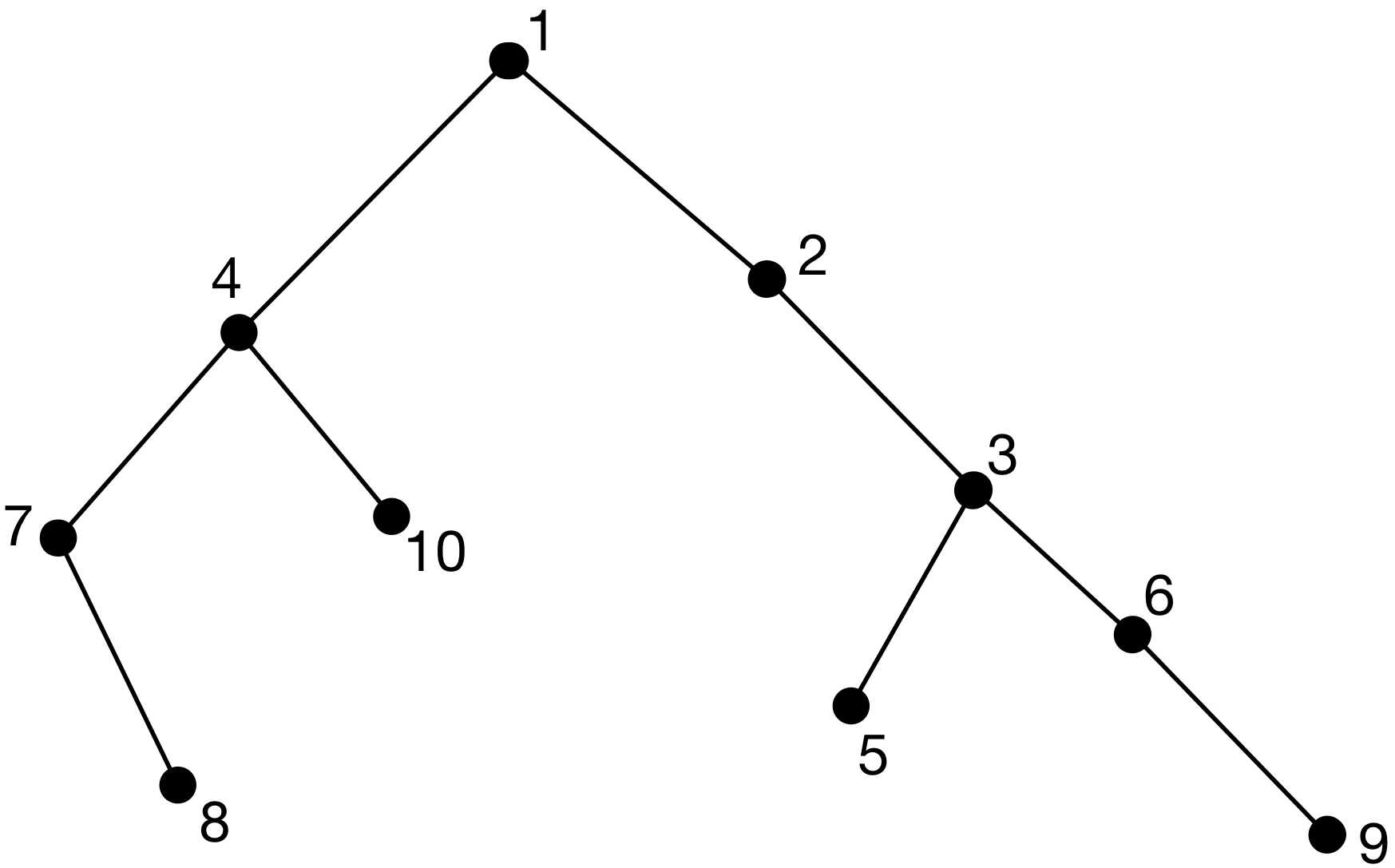}
\caption{(Left) André tree $T_{\s}\in\T_{10}^1$ with $\s=4\,7\,2\,8\,1\,3\,6\,5\,9\,10\in \A_{10}^1$
and the transformed tree $\phi_{5}(T_\s)=\omega(\tau)$ with $\tau={7\,8\,4\,10\,1\,2\,5\,3\,6\,9}.$}\label{example: phi}
\end{figure}
\begin{lemma}\label{lem:comm}
 Let  $T_{\s}\in\T^1_n$ with  $\s=\s_1\dots\s_n\in\A_n^1$.
 Then,  for any $i,j\in S_{T}$, the operators $\phi_i$ and $\phi_j$ commute, namely, 
 \begin{equation}\label{equ-commute}   \phi_i(\phi_j(T_{\s}))=\phi_j(\phi_i(T_{\s})).
 \end{equation}
\end{lemma}
\begin{proof}
If $\phi_i$ (or $\phi_j$) is an identity, then Eq.~\eqref{equ-commute} is trivial. 
Otherwise, 
without loss of generality, 
we assume that $\s_i>\s_j$. 
If  $\s_i$ is not a descendant of $\s_j$,  then $\phi_i$ and $\phi_j$ operate on  two disjoint vertex subsets of $T_{\s}$ and the commutativity is obvious. 
In what follows, if $a$ is an interior vertex of $T_{\s}$, we denote by  $l_{a}$ and $r_{a}$ the left and right successors of $a$.
We suppose  that $\s_i$ is a descendant of $l_{\s_j}$, and
\begin{itemize}
    \item the vertex set  of $T(l_{\s_j})$ (resp. $T(r_{\s_j})$) is $A_{l_{\s_j}}:=\{a_1<a_2<\cdots<a_m\}$ (resp. $B_{r_{\s_j}}:=\{b_1<b_2<\cdots<b_l\}$) with 
    $a_1=l_{\s_j}$ (resp. $b_1=r_{\s_j}$) and $a_m<b_l$;
    \item the vertex set of $T(l_{\s_i})$ (resp.  $T(r_{\s_i})$) is  $A_{l_{\s_i}}:=\{a_{t_1}<a_{t_2}<\cdots < a_{t_p}\}$  (resp.  $A_{r_{\s_i}}:=\{a_{h_1}<a_{h_2}\cdots <a_{h_q}\}$) with $a_{t_1}=l_{\s_i}$ (resp.  $a_{h_1}=r_{\s_i}$).  Note that these two sets are  subsets of $A_{l_{\s_j}}$ and $a_{t_p}<a_{h_q}$.
\end{itemize}

\begin{enumerate}
    \item 
We firstly consider $\phi_i(\phi_j(T_{\s}))$.
By definition, the vertices  of the $T(l_{\s_j})$ and $T(r_{\s_j})$ are changed according to the following substitutions,  
\begin{equation}\label{phi_j} 
\begin{array}{cccccc}
    a_1 & a_2 &a_3 & \cdots &a_{m-1} &a_m\\
    \downarrow & \downarrow &\downarrow& \cdots &\downarrow &\downarrow\\
    a_2 & a_3 & a_4& \cdots &a_m &b_l
\end{array}\,\text{and}\,
\begin{array}{cccccc}
    b_1 & b_2 &b_3 & \cdots &b_{l-1} &b_l\\
    \downarrow & \downarrow &\downarrow& \cdots &\downarrow &\downarrow\\
    a_1 & b_1 & b_2& \cdots &b_{l-2} &b_{l-1}
\end{array}.
\end{equation}
In particular, as $A_{l_{\s_i}}$ and $A_{r_{\s_i}}$ are subsets of $A_{l_{\s_j}}$, under $\phi_j$ the vertices of $T(l_{\s_i})$ and $T(r_{\s_i})$ are changed as follows,
\begin{equation}
    \begin{array}{ccccc}
    a_{t_1} & a_{t_2} & \cdots &a_{t_{p-1}} &a_{t_p}\\
    \downarrow & \downarrow &\cdots &\downarrow &\downarrow\\
    a_{t_1+1} & a_{t_2+1} & \cdots &a_{t_{p-1}+1} & a_{t_{p}+1}
\end{array}\,\text{and}\,
\begin{array}{ccccc}
    a_{h_1} & a_{h_2} & \cdots &a_{h_{q-1}} &a_{h_q}\\
    \downarrow & \downarrow &\cdots &\downarrow &\downarrow\\
    a_{h_1+1} & a_{h_2+1} & \cdots &a_{h_{p-1}+1} & a_{h_{q}+1}
\end{array},
\end{equation}
where if $a_{h_q}=a_m$, then $a_{h_q+1}=b_l$.
Next, proceeding by $\phi_i$, the veritices of $\phi_j(T(l_{\s_i}))$ and $\phi_j(T(r_{\s_i}))$ are changed  according to
\begin{equation}\label{phi-j-i}
    \begin{array}{ccccc}
     a_{t_1+1} & a_{t_2+1} & \cdots &a_{t_{p-1}+1} & a_{t_{p}+1}\\
    \downarrow & \downarrow &\cdots &\downarrow &\downarrow\\
    a_{t_2+1} & a_{t_3+1} & \cdots &a_{t_{p}+1} & a_{h_{q}+1}
\end{array}\,\text{and}\,
\begin{array}{ccccc}
    a_{h_1+1} & a_{h_2+1} & \cdots &a_{h_{q-1}+1} & a_{h_{q}+1}\\
    \downarrow & \downarrow &\cdots &\downarrow &\downarrow\\
    a_{t_1+1} & a_{h_1+1} & \cdots &a_{h_{q-2}+1} & a_{h_{q-1}+1}
\end{array}.
\end{equation}
\item
Now, we consider $\phi_j(\phi_i(T_{\s}))$. Under $\phi_i$, the vertices of $T(l_{\s_i})$ and $T(r_{\s_i})$ are changed according to the following substitutions, 
\begin{equation}
    \begin{array}{ccccc}
    a_{t_1} & a_{t_2} & \cdots &a_{t_{p-1}} &a_{t_p}\\
    \downarrow & \downarrow &\cdots &\downarrow &\downarrow\\
    a_{t_2} & a_{t_3} & \cdots &a_{t_{p}} & a_{h_{q}}
\end{array}\,\text{and}\,
\begin{array}{ccccc}
    a_{h_1} & a_{h_2} & \cdots &a_{h_{q-1}} &a_{h_q}\\
    \downarrow & \downarrow &\cdots &\downarrow &\downarrow\\
    a_{t_1} & a_{h_1} & \cdots &a_{h_{q-2}} & a_{h_{q-1}}
\end{array}.
\end{equation}
Note that the remaining  vertices are not changed under $\phi_i$. Next, we proceed by $\phi_j$, the vertices of $\phi_i(T(l_{\s_j}))$ and $\phi_i(T(r_{\s_j}))$ are changed by~\eqref{phi_j}, then the vertices of $\phi_i(T(l_{\s_i}))$ and $\phi_i(T(r_{\s_i}))$ are changed as follows  
\begin{equation}\label{phi-i-j}
    \begin{array}{ccccc}
    a_{t_2} & a_{t_3} & \cdots &a_{t_{p}} & a_{h_{q}}\\
    \downarrow & \downarrow &\cdots &\downarrow &\downarrow\\
    a_{t_2+1} & a_{t_3+1} & \cdots &a_{t_{p}+1} & a_{h_{q}+1}
\end{array}\,\text{and}\,
\begin{array}{ccccc}
    a_{t_1} & a_{h_1} & \cdots &a_{h_{q-2}} & a_{h_{q-1}}\\
    \downarrow & \downarrow &\cdots &\downarrow &\downarrow\\
    a_{t_1+1} & a_{h_1+1} & \cdots &a_{h_{q-2}+1} & a_{h_{q-1}+1}
\end{array}.
\end{equation}
\end{enumerate}
Comparing the bottom rows~\eqref{phi-j-i} and~\eqref{phi-i-j}, we prove~\eqref{equ-commute}.
\end{proof}

\begin{figure}[tp]
\centering
\includegraphics[scale=0.1]{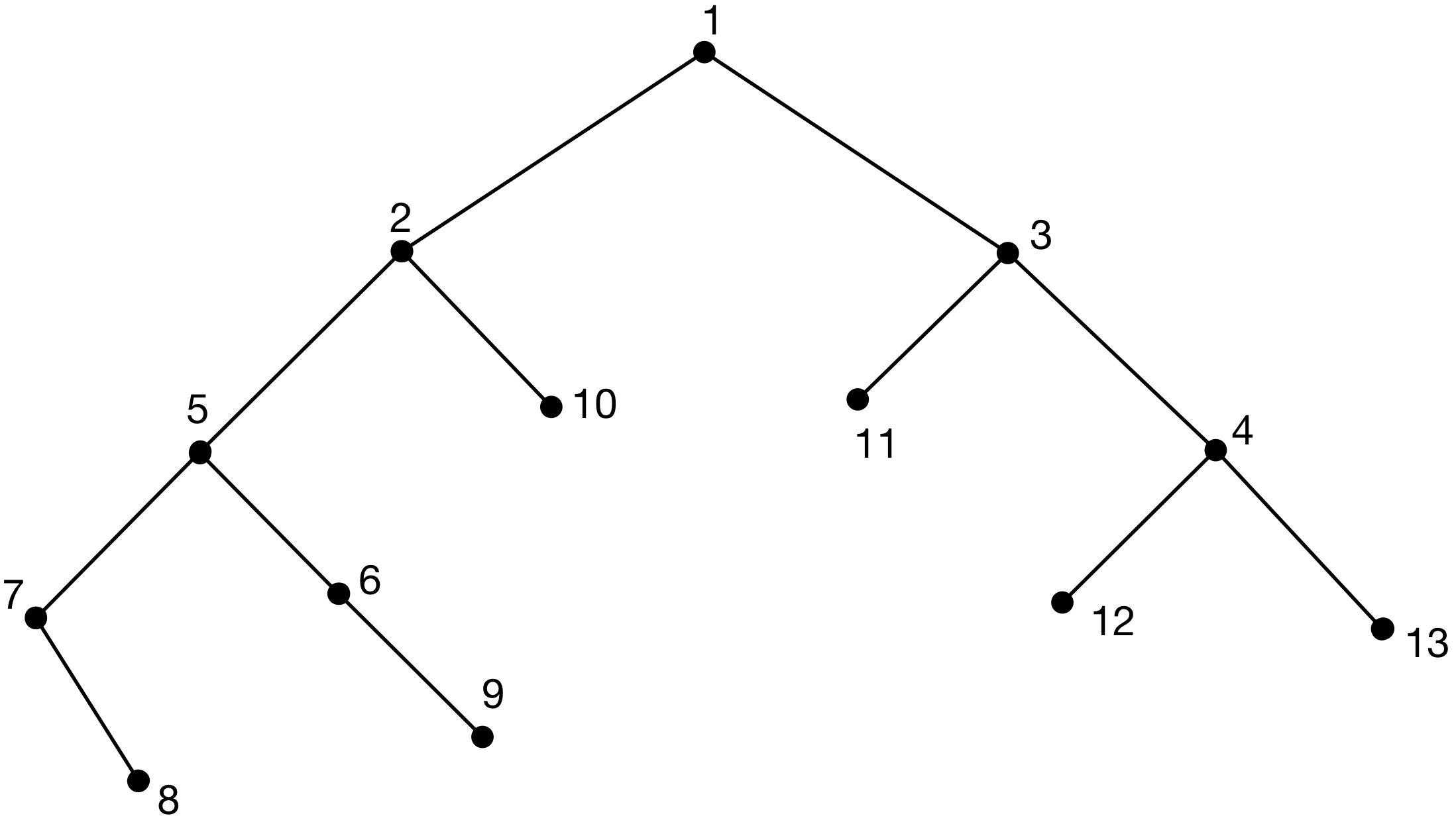}
\includegraphics[scale=0.1]{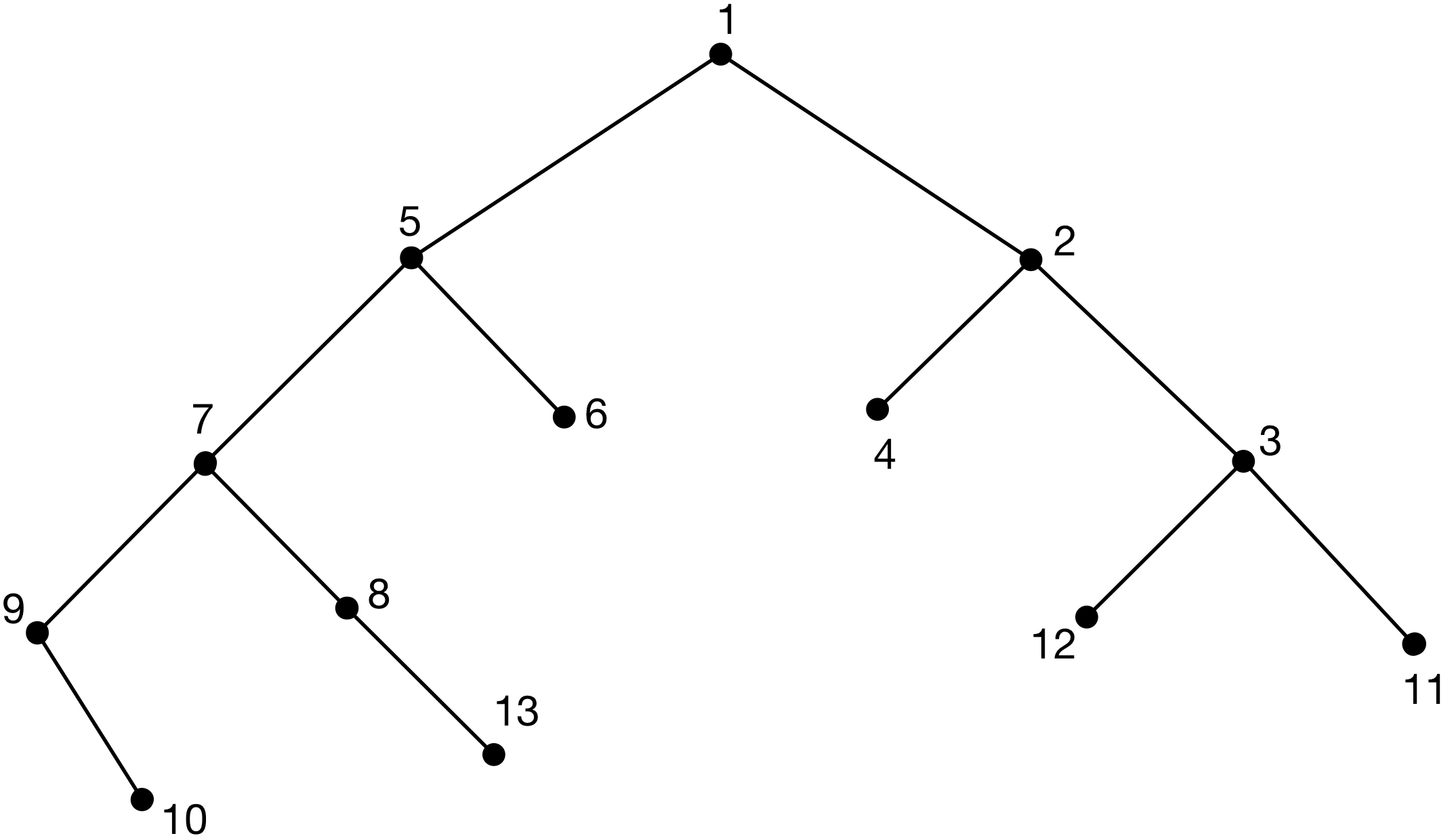}
\caption{(Left) the André tree $T_{7\,8\,5\,6\,9\,2\,10\,1\,11\,3\,12\,4\,13}\in\T_{13}^1$; (Right) the transformed tree $T_{9\,10\,7\,8\,13\,5\,6\,1\,4\,2\,11\,3\,12}\in\T_{13}^2$ under $\Psi=\prod_{i\in S_{T}}\phi_i$, where $S_{T}=\{3,6,8,10,12\}$.}\label{example: mapping}
\end{figure} 

By lemma~\ref{lem:comm}, we can define the operator $\Psi:=\prod_{i\in S_T}\phi_i$ for any $T\in \T_n^1$.

\begin{lemma}
    The operator $\Psi$ is a bijection from $\T_n^1$ to $\T_n^2$.  Moreover, 
    if $T_{\s}=\omega(\s)$ with  $\s\in\A_{n}^1$, then
\begin{equation}\label{thm:two-trees}
(\leaf, \rsho',\rsho)\,T_{\s}=(\leaf, \rsho',\rsho)\,\Psi(T_{\s}).
\end{equation} 
\end{lemma}
\begin{proof}
    
   By Definition~\ref{def:andre:trees} and the construction of $\Psi$,  for each interior vertex  of $\Psi(T_{\s})$ with two successors, its left successor is greater than its right successor, i.e., $\Psi(T_{\s})\in\T_{n}^2$. 

We first show that  $\Psi$ is an injection.  Let   ${T_{\widehat{\s}}}:=\Psi(T_{\s})\in\T_n^2$, with  
$$S_{\widehat{T}}=\{i: \widehat{\s}_i\,\, \text{is a vertex of  $T_{\widehat{\s}}$ with 2 successors} \}.$$ 
If $i\in S_{T}$, let $\widehat{\s}_j$  and $\widehat{\s}_k$ be the left and right successors of $\widehat{\s}_i$.  
  We define $\phi_i^{-1}$  as follows:
\begin{itemize}
    \item 
if $\max (T(\widehat{\s}_j))<\max (T(\widehat{\s}_k))$, then $\phi_i^{-1}T_{\widehat{\s}}=T_{\widehat{\s}}$; 
\item if $\max (T(\widehat{\s}_j))>\max (T(\widehat{\s}_k))$, then exchange $\widehat{\s}_k$ with  $\max(T(\widehat{\s}_j))$, 
and rearrange the vertices of $T(\widehat{\s}_j)$ and $T(\widehat{\s}_k)$,  so that they keep their same relative orders. 
\end{itemize}
 By similar argument  as in the proof of Lemma~\ref{lem:comm}, we can prove that operators $\phi_{i}^{-1}$ and $\phi_j^{-1}$ commute. 
 Since  the operator $\Psi$ just permutes the labels of vertices of $T_{\s}$, we have $S_T=S_{\widehat{T}}$.
Define $\Psi^{-1}:=\prod_{i\in S_{T}}\phi_i^{-1}$, then  
$$
\Psi^{-1}(\Psi(T_{\s}))=\prod_{i\in S_{T}}(\phi_i^{-1}\circ\phi_i)(T_{\s})=\mathbf{id}(T_{\s})=T_{\s}.
$$
Thus, $\Psi$ is an injection from $\T_{n}^1$ to $\T_{n}^2$. As $|\T_{n}^1|=|\T_{n}^2|=E_n$, see~\cite{Dis14,FSch73}, hence $\Psi$ is a bijection. By the above argument, Eq.~\eqref{thm:two-trees} is clear.
\end{proof}

\begin{proof}[\textbf{Proof of Theorem~\ref{main-theorem1}~\eqref{combinatorial interpretation of d-andre2}}]
 It is clear that $\Phi:=\omega^{-1}\circ \Psi\circ \omega$ is a bijection from
 $\A_n^1$ to $\A_n^2$.
 Combining~\ref{thm:two-trees} with Proposition~\ref{prop:statistics-keep}, we prove~\eqref{combinatorial interpretation of d-andre2}. 
\end{proof}

\begin{example}
If $\s=7\,8\,5\,6\,9\,2\,10\,1\,11\,3\,12\,4\,13\in \A_{13}^1$, then 
   $T_{\s}\in\T_{13}^1$, see Fig.~\ref{example: mapping}. As  
   $S_{T}=\{3,6,8,10,12\}$,  let  $\Psi=\prod_{i\in S_{T}}\phi_i$ we obtain  $\Psi(T_\s)=T_\tau\in\T_{13}^2$ and finally
   $$
   \tau=\omega^{-1}\circ \Psi\circ \omega(\s)={9\,10\,7\,8\,13\,5\,6\,1\,4\,2\,11\,3\,12}\in \A_{13}^2.
   $$  
\end{example}

\begin{remark}
    Foata and Han gave another bijetion from $\A_n^1$ to $\A_n^2$ in~\cite[Section~3]{FH16}, but their bijection does not keep the triple statistic $(\des,\rlminda,\rlmin)$. 
\end{remark}
\section{Link to successions}

Let $\sigma\in\S_n$, 
 an index $i\in [n-1]$ is  called a
\begin{itemize}
\item  \emph{succession} ($\suc$) of $\s$ if $\sigma(i+1)=\sigma(i)+1$;
\item  \emph{big ascent} ($\basc$)  of $\s$ if $\sigma(i+1)\geq \sigma(i)+2$.
\end{itemize}
 Let $\Suc(\sigma)$ and $\Basc(\sigma)$  be the sets of successions and  big ascents
of $\sigma$, respectively.
 Furthermore, we define the sets
\begin{itemize}
\item $\Basc_B(\sigma):
=\{\sigma(i+1)\,|\, \s(i)+1<\s(i+1), \; i\in [n-1]\}$,
\item $\Des_B(\sigma):
=\{\sigma(i+1)\,|\, \s(i)>\s(i+1), \; i\in [n-1]\}$, 
\item $\Suc_B(\sigma):=\{\sigma(i+1)\,|\, 
\s(i)+1=\s(i+1),\, i\in [n-1]\}$.
\end{itemize} 
For  $\sigma\in \S_n$ we define 
\begin{align*}
\widehat{\Fix}(\sigma)&:=\{i\,|\, \sigma(i)=i, \; i\in \{2, \ldots, n\}\},\\
\widehat{\Exc_v}(\sigma)&:=\{\s(i)\,|\, \sigma(i)>i, \; i\in \{2, \ldots, n\}\},\\
\Drop_v(\sigma)&:=\{\s(i)\,|\,  \sigma(i)<i, \; i\in [n]\},
\end{align*}
and denote the corresponding cardinalities by 
$\widehat{\fix}(\sigma)$, $\widehat{\exc}(\sigma)$ and 
$\drop(\sigma)$, respectively.

We recall a bijection 
$\varphi$  in \cite[Section 5]{CHZ97}. 
For $\sigma=\sigma_1\sigma_2\ldots \sigma_n\in \S_n$ we set 
$\s'=\sigma_2\ldots\sigma_n\sigma_1$ with $\s(i)=\s_i$ for $i\in [n]$, and 
define  
$\varphi(\sigma)=\theta_2(\s')$, where $\theta_2$ is the second variant of \textbf{FFT}, see Section~\ref{cyclic valley-hopping}.

\begin{example}
 If $\sigma=1\,4\,2\,8\,3\,6\,7\,5\,9\in { \S}_{9}$, then
$\s'=4\,2\,8\,3\,6\,7\,5\,9\,\mathbf{1}$ and 
\[
\widetilde{\s'}=(1\,4\,3\,8\,9)(5\,6\,7)(2),
\quad \varphi(\sigma)=1\,4\,3\,8\,9\,5\,6\,7\,2.
\]
We have 
$
\widehat{\Exc}_v(\sigma)=\{4,8\}$, $\Drop_v(\sigma)=\{2,3,5\}$, $\widehat{\Fix}(\sigma)=\{6,7,9\}$,
$\Basc_B(\varphi(\s))=\{4,8\}$, 
$\Des_B(\varphi(\s))=\{2,3,5\}$, and $ \Suc_B(\varphi(\s))=\{6,7,9\}$.
\end{example}
A pair of consecutive integers $(i,i+1)$ is called 
a \emph{cyclic succession} of $\s\in \S_n$ if $\s(i)=i+1$ with $i\in [n-1]$.
\begin{theorem}\label{sucfix} The mapping 
$\varphi$ is a bijection on 
${\S}_n$ such that for all $\sigma\in{\S}_n$, $\varphi(\sigma)(1)=\s(1)$ and 
\begin{equation}\label{properties-rho}
    (\widehat{\Exc}_v, \Drop_v, \widehat{\Fix})\,\sigma=(\Basc_B, \Des_B, \Suc_B)\,\varphi(\s).
\end{equation}
    In particular, we have
$
(\widehat{\exc}, \drop, \widehat{\fix})\,\s=(\basc, \des, \suc)\,\varphi(\s).
$
\end{theorem}
\begin{proof}
 Clearly the mappings 
$\s\to \widetilde{\s'}$ and  $\s\to \varphi(\s)$  have the following properties:
\begin{itemize}
\item $i\in \widehat{\Fix}(\s)$ if and only if the pair 
$(i-1, i)$ is a  cyclic succession of  $\widetilde{\s'}$, which is equivalent to $i\in \Suc_B(\varphi(\s))$.
 \item   If $\s(i)\in\Drop_v(\s)$, 
 then $i-1\ge {\s}(i)$. So the pair $(i-1, \;\s(i))$ is a 
 cyclic descent or fixed point of $\widetilde{\s'}$.
  This is equivalent to say that 
  $\s(i)\in \Des_B(\varphi(\sigma))$.
 \item If $\s(i)\in\widehat{\Exc}_B(\s)$, 
 then  $(i-1)+1< \s(i)$. This is equivalent to say that the pair $(i-1, \s(i))$ is a big ascent of $\varphi(\s)$.
 \end{itemize}
  Summarising we have  \eqref{properties-rho}.
\end{proof}
\begin{remark}
The  bijection $\varphi$ gives a combinatorial proof of \cite[Theorem 9]{MQYY24}.
\end{remark}

 Define the polynomials 
$$
F_n(x,y,t\,|\,\alpha)=\sum_{\sigma\in \S_n}x^{\widehat{\exc}(\sigma)}y^{\drop(\sigma)}t^{\widehat{\fix} (\sigma)}\alpha^{\cyc(\sigma)}.
$$
 \begin{theorem}
      For integer $n\geq 0$,  we have
      \begin{subequations}
          \begin{equation}\label{sff}
xF_{n+1}(x,y,t\,|\,\alpha)=A^{\cyc}_{n+1}(x,y,t\,|\,\alpha)+\alpha(x-t)A^{\cyc}_n(x,y,t\,|\,\alpha),
\end{equation}
where $A_0(x,y,t\,|\,\alpha)=1$, and 
\begin{equation}\label{eq:gf-F}
\sum_{n\geq 0}F_{n+1}(x,y,t\,|\,\alpha)\frac{z^n}{n!}=
\alpha e^{(y-t)z}\Big(\frac{(x-y)e^{tz}}{xe^{yz}-ye^{xz}}\Big)^{1+\alpha}.
\end{equation}
      \end{subequations}
\end{theorem}
\begin{proof} 
Note that  
\begin{align*}
    \sum_{\mycom{\sigma\in \S_{n+1}}{\sigma(1)=1}} x^{\exc(\sigma)} y^{\drop(\sigma)}
    t^{\fix(\sigma)}\alpha^{\cyc(\sigma)}
    &=\alpha t A^{\cyc}_{n}(x,y,t\,|\,\alpha),\\
\sum_{\mycom{\sigma\in \S_{n+1}}{\sigma(1)=1}}x^{\widehat{\exc}(\sigma)}y^{\drop(\sigma)}t^{\widehat{\fix} (\sigma)}\alpha^{\cyc(\sigma)}&=\alpha A^{\cyc}_{n}(x,y,t\,|\,\alpha).
\end{align*}
According to  the position of 1 in a permutation,  which is 1 or greater than 1, we have 
\begin{align*}
A^{\cyc}_{n+1}(x, y, t\,|\,\a)
&=\a t A^{\cyc}_{n}(x, y, t\,|\,\a)
+x\sum_{\mycom{\sigma\in \S_{n+1}}{\sigma(1)>1}}x^{\widehat{\exc}(\sigma)}y^{\drop(\sigma)}t^{\widehat{\fix} (\sigma)}\alpha^{\cyc(\sigma)},\\
F_{n+1}(x,y,t\,|\,\alpha)
&=\a A^{\cyc}_{n}(x, y, t\,|\,\a)+
\sum_{\mycom{\sigma\in \S_{n+1}}{\sigma(1)>1}}x^{\widehat{\exc}(\sigma)}y^{\drop(\sigma)}t^{\widehat{\fix} (\sigma)}\alpha^{\cyc(\sigma)}.
\end{align*}
Combining  the above two equations yields 
 recurrence~\eqref{sff}, which is equivalent to 
\begin{equation}\label{eq4}
x\sum_{n\geq 1}F_n(x,y,t\,|\,\alpha)\frac{z^n}{n!}
=\sum_{n\geq 1}A^{\cyc}_{n}(x,y,t\,|\,\alpha)\frac{z^n}{n!}+\alpha (x-t)
\sum_{n\geq 1} A^{\cyc}_{n-1}(x,y,t\,|\,\alpha)\frac{z^n}{n!}.
\end{equation}
Invoking the generating function~\eqref{equ:gen-CS-cyc} and applying 
$\frac{\partial }{\partial z}$ to \eqref{eq4}
we obtain
\[
x\sum_{n\geq 1}F_n(x,y,t\,|\,\alpha)\frac{z^{n-1}}{(n-1)!}=\left(\frac{\partial }{\partial z}+\alpha(x-t)\right)\left(\frac{(x-y)e^{tz}}{xe^{yz}-ye^{xz}}\right)^\alpha,
\]
which can be shown (or by Maple) to be equivalent to 
\eqref{eq:gf-F}.
\end{proof}

\begin{remark} By Theorem~\ref{sucfix}, when $x=\a=1$, 
Eq.~\eqref{eq:gf-F} reduces to  the exponential generating function of the statistic $(\des, \suc)$, which  was first obtained by Roselle~\cite{Ro68}.
\end{remark}

\section{Open problems}
A permutation $\s=\s_1\s_2\ldots $ of an ordered set is said to (\emph{increasing}) \emph{alternating} if $\s_1<\s_2>\s_3<\ldots$. Let $\UD_n$ be the set of up-down permutations in $\S_n$.
A cycle (cyclic permutation) $(\s_1,\s_2,\dots)$ of an ordered set is said to be   \emph{up-down} if $\s_1<\s_2>\s_3<\ldots$ and 
 $\s_1$ is the smallest element.   A permutation $\pi$ is  said to be \emph{cycle-up-down} if it is a product of disjoint up-down cycles. Let $\Delta_n$ be the set of  cycle-up-down permutations of $[n]$. Deutsch and Elizalde~\cite{DE09} proved the following bivariate generating function 
\begin{equation}\label{exp-gen-up-down}
\sum_{n\ge 0}\sum_{\pi\in\Delta_n}t^{\fix(\pi)}\a^{\cyc(\pi)}\frac{z^n}{n!}=\frac{e^{\a(t-1)z}}{(1-\sin z)^{\a}}.
\end{equation}
When $x=1$, Eq.~\eqref{egf-formula-web} reduces to 
\begin{align}\label{egf-formula-web: x=1}
\sum_{n\ge 0}\sum_{\s\in\Web_n}t^{\fix(\s)}\a^{\cyc(\s)}\frac{z^n}{n!}=
\frac{e^{\a(t-1)z}}{(1-\sin z)^{\a}}.
\end{align}
In~\cite[Section~2]{FH16}, Foata and Han constructed a bijection  $\eta:\A_{n+1}^1\to \UD_{n+1}$ preserving the first letter. Applying the mapping $\eta$ to each cycle of $\s\in \Web_n$  immediately yields a bijection 
$\widetilde\eta:\Web_n\to \Delta_n$ such that $(\fix, \cyc)\,\s=(\fix, \cyc)\,\widetilde\eta(\s)$.

Hwang, Jang and Oh~\cite[Corollary 4.4]{HJO23} proved the $t=1$ case of \eqref{egf-formula-web: x=1}.
By Theorem~\ref{main-theorem1},
we also have 
\begin{align}\label{egf-egf-andré: x=1}
\sum_{n\ge 0}\sum_{\s\in\A_{n+1}^i}
t^{\rlminda(\s)}\a^{\rlmin(\s)-1}\frac{z^n}{n!}=
\frac{e^{\a(t-1)z}}{(1-\sin z)^{\a}},
\end{align}
  where $i=1,2$.

 Disanto~\cite[Proposition 1]{Dis14} proved the $t=1$ case of \eqref{egf-egf-andré: x=1} and 
  asked for  a bijection
from $\A_{n+1}^1$ to  $\Delta_n$  enlightening the correspondence between $\rlmin$ and $\cyc$.
In fact, 
combining the bijections $\phi^{-1}: \A_{n+1}^1\to \Web_n$ in Section \ref{mapping} and $\widetilde\eta:\Web_n\to \Delta_n$, we obtain a bijetion  $\Xi:=\widetilde{\eta}\circ\phi^{-1}: \A_{n+1}^1\to \Delta_{n}$ such that 
$(\rlminda,\rlmin)\,\s=(\fix, \cyc)\,\Xi(\s)$ for $\s\in \A_{n+1}^1$.
\begin{problem}
Find a suitable statistic $\widehat{\drop}$ over $\Delta_n$ and  bijection   $\Lambda:\Web_n\to \Delta_n$ such that $\drop\,(\sigma)=\widehat{\drop}\,(\Lambda(\sigma))$ for $\sigma\in \Web_n$.
\end{problem}
In view of the work~\cite{PZ21}, where a $(p,q)$-analogue of \eqref{euler-gamma} is proved,  the following question is natural.
\begin{problem}
Find  a 
combinatorial explanation of
 the identity 
 $\gamma_{n,j}(\a,t)=2^jd_{n,j}(\a,t)$ for  $0\le j\le \lfloor n/2\rfloor$. 
\end{problem}

%
%
%
%
%
\subsection*{Acknowledgement}

 We wish to thank Zhicong Lin and Qiong Qiong Pan for helpful conversations and /or correspondance. The first author was supported by the \emph{China Scholarship Council} (No. 202206220034).

%
%
%


\end{document}